\documentclass[10pt,reqno]{article}
\usepackage{latexsym}
\usepackage{amssymb}
\usepackage{amsmath}
\usepackage[all]{xy}
\usepackage{color}
\usepackage{enumerate,url}

\newtheorem{thm}{Theorem}
\newtheorem{defin}{Definition}
\newtheorem{prop}{Proposition}

\newtheorem{cor}{Corollary}
\newtheorem{rem}{Remark}

\newcommand{\goestozero}{\rightarrow 0}

\newcommand{\beeq}{\begin{equation}}
\newcommand{\eeeqnar}{\end{eqnarray*}}
\newcommand{\beeqnar}{\begin{eqnarray*}}
\newcommand{\eeeq}{\end{equation}}
\newcommand{\beit}{\begin{itemize}}
\newcommand{\eeit}{\end{itemize}}
\newcommand{\bedes}{\begin{description}}
\newcommand{\eedes}{\end{description}}
\newcommand{\been}{\begin{enumerate}}
\newcommand{\eeen}{\end{enumerate}}

\newenvironment{dedication}
  {
   \itshape             
   \raggedleft          
  }

\def\ddu2{{\frac{\partial^2}{\partial u^2}}}
\def\Re{{\mathrm{Re}}}
\def\Im{{\mathrm{Im}}}
\def\ZZ {{\mathbb Z}}

\def\RR {{\mathbb R}}
\def\CC {{\mathbb C}}
\def\PP {{\mathbb P}}
\def\HH {{\mathbb H}}

\def\De{\Delta}

\def\Ga{\Gamma}

\def\la{\lambda}

\def\om{\omega}

   \def\cP{{\cal P}} 
    \def\cW{{\cal W}}

\newenvironment{proof}{\medskip
\noindent{\bf Proof: }}{{\hfill$\square$}{\medskip}}
\begin{document}

\title{Spectral construction of non-holomorphic Eisenstein-type series and their Kronecker limit formula}
\author{James Cogdell \and Jay Jorgenson
\footnote{The second named author acknowledges grant support PSC-CUNY.}\and Lejla Smajlovi\'{c}}
\maketitle

\begin{dedication}
\hspace{4cm}
{Dedicated to Emma Previato, on the occasion of her $65$th birthday.}
\end{dedication}

\begin{abstract}\noindent
Let $X$ be a smooth, compact, projective K\"ahler variety and $D$ be a divisor of a holomorphic
form $F$, and assume that $D$ is smooth up to codimension two.  Let $\omega$ be a K\"ahler form
on $X$ and $K_{X}$ the corresponding heat kernel which is associated to the Laplacian
that acts on the space of smooth functions on $X$.  Using various integral transforms of
$K_{X}$, we will construct a meromorphic function in a complex variable $s$ whose special value at $s=0$
is the log-norm of $F$ with respect to $\mu$.  In the case when $X$ is the quotient
of a symmetric space, then the function we construct is a generalization of the so-called
elliptic Eisenstein series which has been defined and studied for finite volume Riemann surfaces.
\end{abstract}

\section{Introduction}

\subsection{Kronecker's limit formula}
The discrete group $\text{PSL}_{2}(\ZZ)$ acts on the upper half plane $\HH$, and the quotient space $\text{PSL}_{2}(\ZZ)\backslash \HH$
has one cusp which can be taken to be at $i\infty$ by identifying $\text{PSL}_{2}(\ZZ)\backslash \HH$ with its fundamental domain.
Associated to the cusp is a non-holomorphic Eisenstein series ${\cal E}^{\mathrm{par}}_{\infty}(z,s)$ which
initially is defined as a Poincar\'e series for $\Re(s) > 1$ but can be shown to admit a meromorphic continuation to
all $s \in \CC$.  One realization of the classical Kronecker limit formula is the
asymptotic expansion that
\begin{align*}
\mathcal{E}^{\mathrm{par}}_{\infty}(z,s)=
\frac{3}{\pi(s-1)}
-\frac{1}{2\pi}\log\bigl(|\Delta(z)|\Im(z)^{6}\bigr)+C+O_{z}(s-1)
\,\,\,\text{\rm as} \,\,\,
s \rightarrow 1
\end{align*}
where $C=6(1-12\,\zeta'(-1)-\log(4\pi))/\pi$.  An elegant proof of Kronecker's limit formula
can be found in \cite{Siegel80}, though the normalization used in \cite{Siegel80} is slightly
different than in \cite{JST16} from which we quote the above formulation.
The series ${\cal E}^{\mathrm{par}}_{\infty}(z,s)$ has a well-known functional equation which
allows one to restate Kronecker's limit formula as
\begin{equation*}
\mathcal{E}^{\mathrm{par}}_{\infty}(z,s)=
1+ \log\bigl(|\Delta(z)|^{1/6}\Im(z)\bigr)s+O_{z}(s^2)
\,\,\,\text{\rm as} \,\,\,
s \rightarrow 0.
\end{equation*}

There are many results in the mathematical literature which develop and explore analogues of Kronecker's limit formula.
One particularly motivating study is given in \cite{KM79} in which the authors define a non-holomorphic hyperbolic
Eisenstein series ${\cal E}^{\mathrm{hyp}}_{\gamma}(z,s)$ associated to any hyperbolic subgroup, generated by a
hyperbolic element $\gamma$,
of an arbitrary co-finite discrete subgroup $\Gamma$ of $\text{PSL}_{2}(\RR)$.  The Kronecker limit formula obtained in
\cite{KM79} states that the Poincar\'e series which defines ${\cal E}^{\mathrm{hyp}}_{\gamma}(z,s)$ admits a meromorphic continuation to $s\in \CC$ and the value of ${\cal E}^{\mathrm{hyp}}_{\gamma}(z,s)$ at $s=0$ is, in effect, the harmonic one-form
which is dual to the geodesic on $\Gamma \backslash \HH$ associated to $\gamma$.

Abelian subgroups of discrete groups $\Gamma$ which act on $\HH$ are classified as parabolic, hyperbolic and
elliptic, so it remained to define and study non-holomorphic Eisenstein series associated to any elliptic
subgroup of an arbitrary discrete group
$\Gamma$.  Any elliptic subgroup can be viewed as the stabilizer group of a point $w$ on the quotient $\Gamma \backslash \HH$,
where in all but a finite number of cases the elliptic subgroup consists solely of the identity element of $\Gamma$.  One can
envision the notion of a non-holomorphic elliptic Eisenstein series ${\cal E}^{\textrm{ell}}_{w}(z,s)$ which, if the above
examples serve as forming a pattern, will admit a mermorphic continuation and whose special value at $s=0$ will be associated to a
harmonic form of some type specified by $w$.  Indeed, such series were studied in \cite{vP10} and, in fact, the Kronecker limit function
is the log-norm of a holomorphic form which vanishes only at $w$.

\subsection{A unified approach}
The article \cite{JvPS16} developed a unified construction of the hyperbolic, elliptic and parabolic Eisenstein series
mentioned above for any finite volume quotient of $\HH$; of course, if the quotient is compact, then parabolic Eisenstein
series do not exist.
  The goal of \cite{JvPS16} was to devise a means, motivated by a type of pre-trace formula, so that the various
Eisenstein series could be obtained by employing different test functions.  As one would expect, there were numerous
technical considerations which arose, especially in the case when $\Gamma$ was not co-compact.  In the end, one can view
the approach developed in \cite{JvPS16} as starting with a heat kernel, and then undertaking a sequence of integral transforms
until one ends up with each of the above mentioned Eisenstein series.  Whereas the article \cite{JvPS16} did provide a unified
approach to the construction of parbolic, hyperbolic and elliptic Eisenstein series for hyperbolic Riemann surfaces,
the analysis did employ the geometry of $\textrm{\rm SL}_{2}(\RR)$ quite extensively.

\subsection{Our results}
The goal of the present paper is to understand the heat kernel construction of non-holomorphic elliptic Eisenstein series
in a more general setting.  We consider a smooth, complex, projective variety $X$ of complex dimension $N$.  We
fix a smooth K\"ahler metric on $X$, which we denote by the $(1,1)$ form $\omega$.  In general terms, let us now describe the
approach we undertake to define and study what we call elliptic Eisenstein series.

Let $t$ be a positive real variable, and let $z$ and $w$ be points on $X$.
Let $K_{X}(z,w;t)$ be the heat kernel acting on smooth functions on $X$ associated to the Laplacian $\Delta_{X}$ corresponding to
$\omega$; see, for example, \cite{Ch84} or \cite{BGV91}.
One of the key properties of $K_{X}(z,w;t)$ is that it satisfies the heat equation, meaning that
$$
\left(\Delta_{z} + \partial_{t}\right)K_{X}(z,w;t) = 0.
$$
We compute the integral transform in $t$ of $K_{X}(z,w;t)$ after multiplying by a function $G(t,u)$ which satisfies
the differential equation $(\partial_{t} - \partial^{2}_{u})G(t,u) = 0$.  By what amounts to integration by parts, we
get a function $(K_{X}\ast G)(z,w;u)$ which satisfies the equation
$$
\left(\Delta_{z} - \partial^{2}_{u}\right)(K_{X}\ast G)(z,w;u) = 0.
$$
If one formally replaces $u$ by $iu$, one gets the kernel function associated to the wave equation.  However, this substitution
is only formal because of convergence considerations; nonetheless, one is able to use the language of distributions in order to achieve the desired result which is
to obtain a wave kernel $W_{X}(z,w;u)$.  At this point, one would like to integrate the wave kernel against the
test function $(\sinh u)^{-s}$ for a complex variable $s$ to yield, as in \cite{JvPS16}, the elliptic Eisenstein series.  Again,
however, technical problems occur because of the vanishing of $\sinh (u)$ when $u=0$.  Instead, we integrate the wave kernel against
$(\cosh u)^{-s}$, for which there is no such technical issue.  We then replace $s$ by $s+2k$ and sum over $k$, in a manner dictated by,
of all things, the binomial theorem, thus allowing us to mimic the use of $(\sinh u)^{-s}$.  In doing so, we arrive at the
analogue of the elliptic Eisenstein series $E_{X}(z,w;s)$, where $s$ is a complex variable, initially required to have real part $\textrm{\rm Re}(s)$
sufficiently large, $z$ is a variable on $X$, and $w$ is a fixed point on $X$.  Though $w$ may be referred to as the elliptic point, it is,
in the case $X$ is smooth, simply a chosen point on $X$.

As a final step, we let $D$ be the divisor of a holomorphic form $F$ on $X$,
and assume that $D$ is smooth up to codimension two.  We show that the integral of $E_{X}(z,w;s)$ with respect of the metric $\mu_{D}(w)$
on $D$ induced from the K\"ahler form $\omega$ has an expansion in $s$ at $s=0$, and the second order term in $s$ is the log-norm of $F$.
This result is the analogue of the classical Kronecker limit formula.

Thus far, all results are obtained by using the spectral expansion of the heat kernel associated to the Laplacian $\Delta_{X}$.
We can equally well reconsider all of the above steps for the operator $\Delta_{X}-Z$ for any complex number $Z$, in which case
we do not begin with the heat kernel $K_{X}(z,w;t)$ but rather we begin with $K_{X}(z,w;t)e^{-Zt}$.
If, for whatever reason, there is  a means by which we have another expression for the heat kernel,
and also have a compelling reason to choose a specific $Z$, then we may end up with another expression
for $E_{X}(z,w;s)$.  Such a situation occurs when, for instance,  $X$ is the quotient of a symmetric space $G/K$
by a discrete group $\Gamma$.  In this case, the heat kernel can be obtained as the inverse spherical transform of an exponential function.
In that setting, it is natural to take $Z = -\rho^{2}_{0}$ where $\rho_{0}$ is essentially the norm, with respect to the Killing form, of half the sum of the positive roots. (In the notation of Gangoli \cite{Ga68}, our $\rho_0$ would be his $|\rho_*|$.)
Finally, we note that one can, without loss of generality, re-scale the time variable $t$ by a positive constant $c$,
so then all begins with the function $K_{X}(z,w;t/c)e^{-Zt/c}$.  In the development of our results, it will be evident
that it is necessary to both translate the Laplacian $\Delta_{X}$ and re-scale time $t$, where it will become
evident that as long as $\rho^{2}_{0} \neq 0$, then it is natural to take $c = 1/(4\rho^{2}_{0})$, which would have the effect of scaling $\rho_0$ to be $1/2$, or $\rho_0^2$ to be $1/4$. (See section 2.7 below.)

The full development of these considerations in all instances would take a considerable amount of time and space, so for the purpose of
the present article we will focus on the results obtainable by considering the spectral decomposition of the heat kernel
in the case of a compact K\"ahler variety $X$.  However, it is possible to give an indication of what will follow
when an additional expression for the heat kernel is available.  For instance, if $X$ is an abelian variety, we obtain
an expression for the heat kernel on $X$ by viewing $X$ as a complex torus.  As an example of our analysis, we can take
$D$ to be the divisor of the Riemann theta function $\theta$, so then our construction expresses the log-norm of the
Riemann theta function $\theta$ as a type of Kronecker limit function.

\subsection{Outline of the paper}
The article is organized as follows.  In section 2 we establish notation and recall some known results.
In section 3, we define the wave distribution associated to a certain space of smooth functions on $X$.  In
section 4 we apply the wave distribution to the test function $\cosh^{-(s-\rho_{0})}(u)$, for a suitably chosen constant
$\rho_{0}$, yielding a function $K_{X;\rho_{0}^{2}}(z,w;s)$.
In section 5 we define two series formed from $K_{X;\rho_{0}^{2}}(z,w;s)$, one producing a formula for the resolvent kernel $G_{X;\rho_{0}^{2}}(z,w;s)$,
which is the integral kernel that inverts the operator $\Delta_{X}+s(s-\rho_{0})$.  The second series $E_{X;\rho_{0}^{2}}(z,w;s)$
is the analogue of the elliptic Eisenstein series.  The analogue of Kronecker's limit formula is given in section 6.  Finally, in
section 7, we conclude with some examples.  In our opinion, each example is of independent interest.  Admittedly, the discussion in section 7 is somewhat speculative; however, we elected to include the discussion in an attempt to illustrate some of the directions
we believe our results can apply.

\it In an unavoidable mishap of notation, the heat
kernel on $X$ is denoted by $K_{X}(z,w;t)$, and the function obtained by applying the wave distribution to $\cosh^{-(s-\rho_{0})}(u)$
is $K_{X;\rho_{0}^{2}}(z,w;s)$.  Similarly, $\Gamma$ will sometimes signify the Gamma function and sometimes signify a discrete
group acting on a symmetric space.  In each case, the meaning will be clear from the context of the discussion. \rm

\section{Background material}
\label{s.Background}

In this section we establish notation and state certain elementary results which
will be used throughout the article.  The contents in this section are given in no particular
order of importance.

\subsection{Stirling's approximation}

Stirling's approximation for the logarithm $\log \Gamma(s)$ of the classical gamma function is
well-known, and we will use the form which states that
\begin{equation}\label{e.Stirling}
\log \Gamma(s) = s\log (s) - s +\frac{1}{2}\log (2\pi/s) + \sum\limits_{n=1}^{M}\frac{B_{2n}}{2n(2n-1)s^{2n-1}} + h_M(s),
\end{equation}
where $B_{n}$ is the $n$-th Bernoulli number and $h_M(s)$ is a holomorphic function in the half-plane $\Re(s)\gg0$ and
$h_M(s)=O_{M}(s^{-2M-1})$ as $s \rightarrow \infty$.
The proof of (\ref{e.Stirling}) is found in various places in the literature; see, for example, \cite{JLa93}.
Going further, the proof from \cite{JLa93} extends to show that one can, in effect, differentiate the above
asymptotic formula.  More precisely, for any integer $\ell \geq 0$, one has that
\begin{equation}\label{e.Stirling_derivative}
\partial_{s}^{\ell} \log \Gamma(s) = \partial_{s}^{\ell}\left(s\log (s) - s +\frac{1}{2}\log (2\pi/s) + \sum\limits_{n=1}^{M}\frac{B_{2n}}{2n(2n-1)s^{2n-1}}\right) + \partial_{s}^{\ell}h_M(s),
\end{equation}
where $\partial_{s}^{\ell}h_M(s)=O_{M,\ell}(s^{-2M-\ell-1})$, as $s \rightarrow \infty$.

We will use the notational convenience of the Pochhammer symbol $(s)_{n}$, which
is defined as
$$
(s)_{n} := \frac{\Gamma(s+n)}{\Gamma(s)}.
$$

\subsection{Elementary integrals}

For any real number $r \in \RR$ and complex number $\nu$ with $\textrm{Re}(\nu)>0$, we
have, from $3.985.1$ of \cite{GR07} the integral formula
\beeq
\label{e.GR07-3.9815.1}
\int_0^\infty \cos(u r) \cosh^{- \nu}(u) \, du =
\frac{2^{\nu - 2} }{\Ga(\nu)} \Ga\left(  \frac{\nu - ir}{2}\right)  \Ga\left(  \frac{\nu + ir}{2}\right).
\eeeq
If $r = 0$, then we get the important special case that
\beeq
\label{e.GR07-3.512.2}
\int_0^\infty  \cosh^{- \nu}(u) \, du
=
\frac{2^{\nu - 2} \Ga^{2}(\nu/2)}{\Ga(\nu)},
\eeeq
which is stated in $3.512.1$ of \cite{GR07}.  Additionally, we will use that
for any $r\in \CC$ with $\textrm{Re}(r^{2})>0$ and $u\in \CC$ with $\textrm{Re}(u)> 0$, one has that
\beeq\label{heat_to_Poisson}
\frac{u}{\sqrt{4\pi}} \int_0^\infty
 e^{-r^2 t} e^{-u^2/(4t)} t^{-1/2} \, \frac{dt}{t} = e^{|r| u}.
\eeeq

For any real valued function $g$ and $r\in \mathbb C$, we define
$H(r, g)$ as
\beeq \label{d.Hrg}
H(r,g) := 2 \int_0^\infty \cos (ur) g(u) \, du.
\eeeq

\noindent
This is a purely formal definition; the conditions on $g$ under which we
consider $H(r,g)$ are stated in section 3 below.

\subsection{An asymptotic formula}

For the convenience of the reader, we state here a result from page 37 of \cite{Er56}.
Let $(\alpha , \beta) \subset \RR$.  Let $g$ be a real-valued continuous function, and $h$ be
a real-valued continuously differentiable function, on $(\alpha, \beta)$
such that the integral
$$
\int\limits_{\alpha}^{\beta}g(t)e^{xh(t)}dt
$$
exists for sufficiently large $x$.  Assume there is an $\eta > 0$ such that $h'(t) < 0$
for $t \in (\alpha, \alpha+\eta)$.  In addition, for some $\epsilon > 0$, assume that
$h(t) \leq h(\alpha) - \epsilon$ for $t \in (\alpha+\eta, \beta)$.
Suppose that
$$
h'(t) = -a(t-\alpha)^{\nu-1} + o((t-\alpha)^{\nu-1})
\,\,\,\,\,
\textrm{and}
\,\,\,\,\,
g(t) = b(t-\alpha)^{\lambda-1} + o((t-\alpha)^{\lambda-1})
\,\,\,\,\,
\textrm{as $t \rightarrow \alpha^{+}$}
$$
for some positive $\lambda$ and $\nu$.  Then
\begin{equation}\label{e.integral_asymp}
\int\limits_{\alpha}^{\beta}g(t)e^{xh(t)}dt = \frac{b}{\nu}\Gamma(\lambda/\nu)(\nu/(ax))^{\lambda/\nu}e^{x h(\alpha)}\left(1+o(1)\right)
\,\,\,\,\,
\textrm{as $x \rightarrow \infty$.}
\end{equation}
Note that the assumptions on $h$ hold if $\alpha=0$, $h(0)= 0$ and is monotone decreasing, which is the setting in which we will apply the above result.  In this case, we will use that the above integral is $O(x^{-\lambda/\nu})$.

\subsection{Geometric setting}
Let $X$ be a compact, complex, smooth projective variety of complex dimension $N$.
Fix a smooth K\"ahler metric $\mu$ on $X$, which is associated to the K\"ahler $(1,1)$ form $\omega$.
Let $\rho$ denote
a (local) potential for the metric $\mu$. If we choose local holomorphic
coordinates $z_{1}, \dots, z_{N}$ in the neighborhood of a point on $X$, then one can write $\omega$
as
$$
\omega = \frac{i}{2} \sum_{j,k=1}^{N}g_{j,\bar{k}}dz_{j}\wedge d\bar{z}_{k}=\frac{i}{2}\partial_{z}\partial_{\bar{z}}\rho
$$
If $M$ is a subvariety of $X$, the induced metric on $M$ will be denoted by $\mu_{M}$.  In particular, the
induced metric on $X$ itself is $\mu_{X}$, which we will simply write as $\mu$.  In a slight abuse of notation,
we will also write $\mu_{M}$, or $\mu$ in the case $M=X$, for the associated volume form against which one
integrates functions.

The corresponding Laplacian $\Delta_X$ which acts on smooth functions on $X$ is
$$
\Delta_X = -\sum_{j,k=1}^{N}g^{j,\bar{k}}\frac{\partial^{2}}{\partial z_{j}\partial \bar{z}_{k}},
$$
where, in standard notation,  $(g^{j,\bar{k}}) = (g_{j,\bar{k}})^{-1}$; see page 4 of \cite{Ch84}.
An eigenfunction of the Laplacian $\Delta_{X}$ is an \it a priori \rm $C^{2}$ function $\psi_{j}$
which satisfies the equation
$$
\De_X \psi_{j} - \lambda_{j} \psi_{j}= 0
$$
for some constant  $\lambda_{j}$, which is the eigenvalue associated to $\psi_{j}$.  It is well-known that any
eigenfunction is subsequently smooth, and every eigenvalue is greater than zero except when
$\psi_{j}$ is a constant whose corresponding eigenvalue is zero.  As is standard, we
assume that each eigenfunction is normalized to have $L^{2}$ norm equal to one.

Weyl's law asserts that
$$
\#\{\lambda_{j} | \lambda_{j} \leq T\} = (2\pi)^{-2N}\textrm{vol}_{N}(\mathbb{B})\textrm{vol}_{\omega}(X)T^{N} + O(T^{N-1/2})
\,\,\,\,\,\textrm{as $T \rightarrow \infty$}
$$
and
$\textrm{vol}_{\omega}(X)$ is the volume of $X$ under the metric $\mu$ induced by $\omega$, and
$\textrm{vol}_{N}(\mathbb{B)}$ is the volume of the unit ball in $\mathbb{R}^{2N}$.
As a consequence of Weyl's law, one has that for any $\varepsilon > 0$,
\begin{equation}\label{e.series_conv}
\sum\limits_{k=1}^{\infty} \lambda_{k}^{-N-\varepsilon} < \infty;
\end{equation}
see, for example, page 9 of \cite{Ch84}.  The eigenfunction $\psi_j$ corresponding to the eigenvalue $\lambda_j$
satisfies a sup-norm bound on $X$, namely that
\begin{equation}\label{e.sup_norm}
\Vert \psi_{j} \Vert_{\infty} =O_{X}\left(\lambda_{j}^{N/2-1/4}\right);
\end{equation}
see \cite{SZ02} and references therein.

\subsection{Holomorphic forms}\label{s.forms}
By a holomorphic form $F$ we mean a holomorphic section of a power of the canonical bundle $\Omega$ on $X$; see page 146 of \cite{GH78}.
(Note: On page 146 of \cite{GH78}, the authors denote the canonical bundle by $K_{X}$, which we will not since this notation
is being used both for the heat kernel and the function obtained by applying the wave distribution to hyperbolic cosine.)
The weight $n$
of the form equals the power of the bundle of which $F$ is a section. Let $D$ denote the divisor of $F$, and
assume that $D$ is smooth up to codimension two.  In the
case $X$ is a quotient of a symmetric space $G/K$ by a discrete group $\Gamma$, then $F$ is a holomorphic automorphic form
on $G/K$ with respect to $\Gamma$. With a slight gain in generality, and with no increase in complication of the analysis,
we can consider sections of the canonical bundle obtained by considering the tensor product of the canonical bundle with a
flat line bundle.  The K\"ahler form $\omega$ will induce a norm on $F$, which we denote by $\Vert F \Vert_{\omega}$;
see \cite{GH78} for a general discussion as well as section 2 of \cite{JK01}.  We can describe the norm as follows.

As in the notation of section 2 of \cite{JK01}, let $U$ be an element of an open cover of $X$.  Once we trivialize $\Omega$ on $U$,
we can express the form $F$ in local coordinates $z_{1}, \dots, z_{N}$.  Also, we have the existence of a K\"ahler potential
$\rho$ of the K\"ahler form $\omega$.  Up until now, there has been no natural scaling of $\omega$.  We do so now, by scaling
$\omega$ by a multiplicative constant $c$ so that $c\omega$ is a Chern form of $\Omega$; see page 144 of \cite{GH78} as well as
chapter 2 of \cite{Fi18}.   In a slight abuse of notation, we will denote the re-scaled K\"ahler form by $\omega$.

With this scaling of $\omega$, one can
show that $\vert F(z)\vert e^{-n\rho(z)}$ is invariant under change of coordinates; see section 2.3 of \cite{JK01}.  With
this, one defines
\begin{equation}\label{e.Fnorm}
\Vert F \Vert_{\omega}^{2}(z) :=\vert F(z)\vert^{2} e^{-2n\rho(z)},
\end{equation}
where $n$ is the weight of the form.  The formula is local for each $U$ in the open cover, but its invariance implies that the definition extends independently of
the various choices made.  Following the discussion of Chapter 1 of \cite{La88}, the above equation can be
written in differential form as
\begin{equation}\label{ddc}
\textrm{\rm d}\textrm{\rm d}^{c}\log \Vert F \Vert_{\omega}^{2} = n(\delta_{D} - \omega)
\end{equation}
where $\delta_{D}$ denotes the Dirac delta distribution supported on $D$.

K\"ahler metrics have the property that the associated Laplacian of a function does not involve derivatives of
the metric, as stated on page 75 of \cite{Ba06}.   Exercise 1.27.3(a) of \cite{Fi18} states the formula
\begin{equation}\label{e.omega.Laplacian}
\frac{1}{2}\Delta_{X,d}f \omega^{n} = n i \partial \bar{\partial}f \wedge \omega^{n-1}
\end{equation}
where $\Delta_{X,d}$ is the Laplacian stemming from the differential $d$ and $f$ is a smooth function, subject
to the normalizations of various operators as stated in \cite{Fi18}.  As
a corollary of (\ref{e.omega.Laplacian}), one interpret (\ref{ddc}) as asserting that
$\Delta_{X}\log \Vert F \Vert_{\omega}^{2}$ is a non-zero constant away from $D$.

\subsection{The heat and Poisson kernel}

The heat kernel acting on smooth functions on $X$ can be defined formally as
$$
K_X(z, w; t) = \sum_{k = 0}^\infty e^{-\la_k t} \psi_k(z) \overline{\psi_k}(w),
$$
where $\{ \psi_j\}$ are eigenfunctions associated to the eigenvalue $\la_j$.
As a consequence of Weyl's law and the sup-norm bound for eigenfunctions, the series
which defines the heat kernel converges for all $t > 0$ and $z, w \in X$.  Furthermore,
if $z \neq w$, then the heat kernel has exponential decay when $t$ approaches zero;
see page 198 of \cite{Ch84}.

For any $Z \in \CC$ with $\Re(Z) \geq 0$, the {\em translated by $-Z$ Poisson Kernel }
$\cP_{X, -Z}(z, w; u)$, for $z, w \in X$ and $u \in \CC$ with $\Re(u) \geq 0$ is defined by
\beeq
\label{e.DefPoissonKernel}
\cP_{X, -Z}(z, w; u) = \frac{u}{\sqrt{4\pi}} \int_0^\infty
K_X(z, w; t) e^{-Zt} e^{-u^2/(4t)} t^{-1/2} \, \frac{dt}{t}.
\eeeq
The translated Poisson kernel $\cP_{X, -Z}(z, w; u)$ is a
fundamental solution associated to the differential operator $\De_X + Z - \partial^2_u$.  For certain
considerations to come, we will choose a constant $\rho_{0} \geq 0$, which will depend on
the geometry of $X$, and write each eigenvalue of $\De_X $ as $\la_j = \rho_0^2 + t_j^2$.
Thus, we divide the spectral expansion of the heat kernel $K_X$ into two subsets: The finite sum for $\la_j < \rho_0^2$, so then $t_j \in (0, i\rho_0]$, and the sum over $\la_j \geq \rho_0^2$, so then $t_j\geq 0$.
Using (\ref{heat_to_Poisson}), we get the spectral expansion
\beeq
\label{e.PoissonSpectralExpansion}
\cP_{X, -Z}(z, w; u) = \sum_{\la_k < \rho_0^2} e^{-u\sqrt{\la_k + Z} } \psi_k(z) \overline{\psi}_k(w) +
\sum_{\la_k \geq \rho_0^2} e^{-u\sqrt{\la_k + Z} } \psi_k(z) \overline{\psi}_k(w) .
\eeeq

\noindent
By Theorem 5.2 and Remark 5.3 of \cite{JLa03}, $\cP_{X, -Z}(z, w; u) $
admits an analytic continuation to $Z = - \rho_0^2$.  In analogy with
\cite{JvPS16}, we can deduce that
 that the continuation of
$\cP_{X, -Z}(z, w; u) $ for $Z = - \rho_0^2$, with $\Re(u) > 0$ and $\Re(u^2) > 0$
is given by
\beeq
\label{e.AnalPoisson}
\cP_{X, \rho_0^2}(z, w; u) = \sum_{\la_k < \rho_0^2} e^{-u\sqrt{\la_k - \rho^2_0} } \psi_k(z) \overline{\psi}_k(w) +
\sum_{\la_k \geq \rho_0^2} e^{-u{t_k} } \psi_k(z) \overline{\psi}_k(w),
\eeeq
where $\sqrt{\la_k - \rho_0^2} = t_k \in (0, i\rho_0]$ is taken to be the branch of
the square root obtained by analytic continuation through the upper half-plane.

As stated, if $z \neq w$, then the heat kernel has exponential decay as $t$ approaches zero.  From this,
one can show that the Poisson kernel $\cP_{X, -Z}(z, w; u)$ is bounded as $u$ approaches zero for any $Z$.

At this point, we would like to define the (translated by $\rho_0^2$) wave kernel
by defining
$$
W_{X, \rho_0^2}(z, w; u) = \cP_{X, \rho_0^2}(z, w; iu) + \cP_{X, \rho_0^2}(z, w; -iu),
$$
for some branch of the meromorphic continuation of $\cP_{X, \rho_0^2}(z, w; u)$
to all $u \in \CC$.
However, because of convergence issues, we cannot simply replace $u$ by $iu$ in the expression
for the Poisson kernel.  As a result, we define the wave distribution via the spectral expansion
for the analytic continuation of $\cP_{X, \rho_0^2}$ in (\ref{e.AnalPoisson}).

\subsection{An elementary, yet important, rescaling observation}\label{s. rescaling}

By writing the Laplacian as in the beginning of section 2.3, we have established specific conventions
regarding various scales, or multiplicative constants, in our analysis.  However, there is
one additional scaling which could be considered.  Specifically, one could consider the heat equation
$\Delta_{z} + c \partial_{t}$ for any positive constant $c$.  The associated heat kernel would be
$K_X(z, w; t/c)$, if the heat kernel associated to $\Delta_{z} + \partial_{t}$ is $K_X(z, w; t)$.
In doing so, we would replace (\ref{e.DefPoissonKernel}) by
\beeq
\label{e.DefPoissonKernel2}
\cP_{X, -Z}(z, w; u) = \frac{u}{\sqrt{4\pi}} \int_0^\infty
K_X(z, w; t/c) e^{-Zt/c} e^{-u^2/(4t)} t^{-1/2} \, \frac{dt}{t}.
\eeeq
for some positive constant $c$.  In effect, we are changing the parameterize for
the positive real axis $\RR^{+}$ from the parameter $t$ to $t/c$ for any positive constant $c$.
In this manner, we rescale the data from the beginning of our consideration so that when we study the
translation of the heat kernel, we can, provided $\rho_{0}^{2} > 0$, choose $c$ appropriately so that
translation $\rho_{0}^{2}/c$ is always equal to $1/4$.

In the examples we develop, the choice of the
$\rho_{0}^{2}$ will be determined by a ``non-spectral'' representation of the heat kernel, after which
we choose $c = 1/(4\rho_{0}^{2})$, provided $\rho_{0} \neq 0$.  As it turns out, the translation
by $1/4$ matters.  This point will become relevant in section 6 below.

\section{The wave distribution}
\label{s.WaveDistribution}

For $z, w \in X$ and function $g \in C^{\infty}_{c}(\RR^+)$,
we formally define the wave distribution $\cW_{X,\rho_0^2}(z, w)(g)$
applied to $g$ by the series

\begin{equation}
\label{wave d. series}
\cW_{X,\rho_0^2}(z, w) (g) =
\sum_{\la_j \geq 0} H(t_j, g) \psi_j(z) \overline{\psi}_j(w),
\end{equation}
where $H(t_j, g)$ is given by \eqref{d.Hrg} and
$t_{j} = \sqrt{\lambda_{j} - \rho_{0}^{2}}$ if $\lambda_{j} \geq \rho_{0}^{2}$, otherwise
$t_{j} \in (0,i\rho_{0}]$.

\begin{defin}
For $a \in \RR^{+}$ and $m \in \mathbb{N}$, let $S_m^{\prime}(\RR^{+}, a)$ be
the set of Schwartz functions on $\RR^{+}$ with $g^{(k)}(0) = 0$
for all odd integers $k$ with $0\leq k \leq m+1$ and where
$e^{ua}|g(u)|$ is dominated by an integrable function on
$\RR^{+}$.
\end{defin}

The following proposition addresses the question of convergence of \eqref{wave d. series}.

\begin{thm}\label{t.wavedistribution}
Fix $z, w \in X$, with $z \neq w$, there exists a
continuous, real-valued function $F_{z,w}(u)$ on $\RR^+$ and an integer $m$ sufficiently
large such that the following assertions hold.

\begin{enumerate}
\item[(i)] One has that
$F_{z,w}(u) = (-1)^{m+1}\sum_{\la_j < \rho_0 ^2} e^{u\sqrt{\rho_0^2-\la_j }} \cdot t_j^{-(m+1)} \psi_j(z) \overline{\psi}_j(w) + O(u^{m+1})$
as $u \rightarrow \infty$.
\item[(ii)]
For any non-negative integer $j \leq m$, we have the bound $\partial_{u}^{j}F_{z,w}( u)  = O(u^{m+1 -j} )$ as $u \rightarrow 0^{+}$.
\item[(iii)]
For any $g\in S_{m}^{\prime }(\mathbb{R}^{+},\rho_{0})$ such that
$\partial_{u}^{j}g(u) \exp (\rho_{0}u)$ has a limit as $u\to\infty$ and is bounded by some integrable function
on $\mathbb{R}^{+}$ for all non-negative integers $j \leq m+1$,
we have
\begin{equation} \label{Wave via Poisson}
\mathcal{W}_{X,\rho_{0}^{2}}(z,w)(g)=
\int\limits_{0}^{\infty}F_{z,w}(u) \partial_{u}^{m+1}g(u)du.
\end{equation}
\end{enumerate}
The implied constants in the error terms in statements (i) and (ii) depend on $m$ and the distance between $z$ and $w$.
\end{thm}

\begin{proof}  Choose an integer $m \geq 4N+1$, where $N$ is the complex dimension of $X$.
To begin, we claim the following statement:
For every integer $k$ with $0 \leq k \leq m$, there is a polynomial $h_{k,m}(x)$
of degree at most $m$ such that
$$
h_{k,m}(\sin(x)) = \frac{x^{k}}{k!} + O_{m}(x^{m+1})
\,\,\,\,\,
\text{\rm as $x \rightarrow 0$.}
$$
Indeed, one begins by initially setting $h_{k,m}^{(0)}(x) = x^{k}/k!$.
The function $h_{k,m}^{(0)}(\sin(x))$
has a Taylor series expansion near zero of the form
$$
h_{k,m}^{(0)}(\sin(x)) = \frac{x^{k}}{k!} + c_{\ell}x^{\ell} +O_{\ell}(x^{\ell+1})
\,\,\,\,\,
\text{\rm as $x \rightarrow 0$}
$$
for some real number $c_{\ell}$ and integer $\ell \geq k+1$.  Now set $h_{k,m}^{(1)}(x)$
to be $h_{k,m}^{(1)}(x)= h_{k,m}^{(0)}(x) - c_{\ell}x^{\ell}$ so then
$$
h_{k,m}^{(1)}(\sin(x)) = \frac{x^{k}}{k!} + c_{p}x^{p} + O_{p}(x^{p+1})
\,\,\,\,\,
\text{\rm as $x \rightarrow 0$}
$$
for some real number $c_{p}$ and integer $p \geq \ell+1 \geq k+2$.
One can continue
to subtract multiplies of monomials of higher degree, thus further reducing the order of the error term until the
claimed result is obtained; it is elementary to complete the proof of the assertion with the appropriate proof by induction argument.

Having proved the above stated assertion, we then have for any $\zeta\in\CC\setminus\{0\}$ that
\begin{equation} \label{asymptotic t to 0}
e^{-t\zeta} - \sum_{k=0}^{m} h_{k,m}(\sin(t))(-\zeta)^{k}  = O_{\zeta}(t^{m+1})
\,\,\,\,\,
\text{\rm as $t \rightarrow 0$.}
\end{equation}
For $t>0$ we define
\begin{equation}
P_m(t,\zeta ) := \frac{e^{-t\zeta} - \sum_{k=0}^{m} h_{k,m}(\sin(t))(-\zeta)^{k}}{(-t)^{m+1}}
\end{equation}
and set $P_m(0,\zeta )= \lim_{t \to 0} P_m(t,\zeta )$; the existence of this limit is ensured by \eqref{asymptotic t to 0}.
For $\Re(\zeta) \geq 0$, we have $P_m(t,\zeta ) = O(t^{-m-1})$ as $t\to \infty$. Hence, the bounds for the eigenvalue growth
(\ref{e.series_conv}) and sup-norm for eigenfunctions (\ref{e.sup_norm}), together with the choice of $m$, imply that the series
\begin{equation}
\label{e.defFtilde}
\tilde{F}_{z, w}(\zeta) = \sum_{\la_j \geq 0} P_m(t_j,\zeta ) \psi_j(z) \overline{\psi}_j(w)
\end{equation}
converges uniformly and absolutely on $X$ for $\zeta$ in the closed half-plane $\Re(\zeta)\geq 0$.

Furthermore, any of the first $m+1$ derivatives of $\tilde{F}_{z, w}(\zeta)$ in $\zeta$ converges uniformly and absolutely when $\Re(\zeta) > 0$.
This allows us to differentiate the series above term by term,
and by doing so $m+1 $ times and using that
$$
\frac{d^{m+1}}{d\zeta^{m+1}}P_m(t_j,\zeta ) = e^{-t_{j}\zeta}
$$
we conclude that for $\Re(\zeta) > 0$, one has the identity
\begin{equation}
\label{e.ftilde}
\frac{d^{m+1}}{d\zeta^{m+1}} \tilde{F}_{z,w}(\zeta) = \cP_{M, \rho_0 ^2}(z,w; \zeta),
\end{equation}
where $\cP_{M, \rho_0 ^2}$ is defined in (\ref{e.AnalPoisson}).
Set $\cP^{(0)}(z, w;\zeta) = \cP_{X,\rho_0 ^2}(z,w; \zeta)$, and define inductively for $\Re(\zeta) > 0$
the function
$$
\cP^{(k)}(z, w; \zeta) =  \int_0^\zeta \cP^{(k-1)}(z, w; \xi) \, d\xi,
$$
where the integral is taken over a ray contained in the upper half plane.
Note that
\begin{equation}
\label{e.asympPn}
\cP^{(k)}(z, w; \zeta) = O_{z,w,k}(\zeta^k) \; \;  \text{as} \; \; \zeta \goestozero.
\end{equation}
From (\ref{e.ftilde}), we have that
\begin{equation}
\label{e.seventeen}
\cP^{(m+1)}(z, w; \zeta)  - \tilde{F}_{z,w}(\zeta)  = q_m(z, w; \zeta),
\end{equation}
where $q_m(z, w; \zeta)$
is a degree $m$ polynomial in $\zeta$ with coefficients which depend on $z$ and $w$.
Using this, for $u\in  \RR^{+}$ we define
\begin{equation}
\label{e.definitionofF}
F_{z,w}(u) =
\frac{1}{2i} \left[  \left(  \tilde{F}_{z, w}(iu) + q_m(z, w; iu) \right)
- \left(  \tilde{F}_{z, w}(-iu) + q_m(z, w; -iu) \right)
\right]
\end{equation}

Assertions (i) and (ii) follow immediately from the above construction of $F$
and its relation to the Poisson kernel.
Since the expansion (\ref{e.defFtilde}) converges uniformly for $\Re(\zeta) =0$,
property (iii) will follow directly from (i), (ii) and $(m+1)$ term-by-term integration by parts.
\end{proof}

\section{A basic test function}

The building block for our Kronecker limit formula is obtained by applying the wave distribution
to the function $\cosh^{-(s-\rho_{0})}(u)$.  As stated above, we will choose $\rho_{0}$
depending on the geometry of $X$ and then re-scale the time
variable in the heat kernel so that ultimately we have either $\rho_{0}=0$ or $\rho_{0}=1/2$.
For the time being, let us work out the results for a general $\rho_{0}$.

\vskip .10in
\begin{prop}
\label{p.spectralW}
For $s \in \CC$ with $\Re(s) > 2\rho_{0}$, the wave distribution of $g(u) = \cosh^{-(s-\rho_{0})}(u)$ exists
and admits the spectral expansion
\begin{equation} \label{wave spectral cosh}
\cW_{X,\rho_{0}^{2}}(z, w)(\cosh^{-(s-\rho_{0})}) = \sum_{\la_j \geq 0}  c_{j, (s-\rho_0)} \psi_j(z) \overline{\psi}_j(w),
\end{equation}
where
\begin{equation}
\label{e.cu}
c_{j,(s-\rho_0)}  =
\frac{2^{s-\rho_{0} - 1} }{\Ga(s-\rho_{0})} \Ga\left(  \frac{s-\rho_{0} - it_j}{2}\right)  \Ga\left(  \frac{s-\rho_{0} + it_j}{2}\right).
\end{equation}
Furthermore, for any $z,w\in X$, the series \eqref{wave spectral cosh} converges absolutely and uniformly in $s$ on any compact subset of the half-plane
$\Re(s) > 2\rho_0$.
\end{prop}

\begin{proof}
If $\Re(s) > 2\rho_0$, then the conditions of Theorem \ref{t.wavedistribution} apply.  The spectral coefficients
are computed using (\ref{e.GR07-3.9815.1}) and (\ref{e.GR07-3.512.2}).  Finally, Stirling's formula (\ref{e.Stirling_derivative}) implies that
the factor \eqref{e.cu} decays exponentially as $t_j \to \infty$.  When combined with the sup-norm bound (\ref{e.sup_norm})
on the eigenfunctions $\psi_j$, the assertion regarding uniform convergence follows.
\end{proof}

\vskip .10in
\begin{cor}\label{Lemma K-properties}
For $z,w\in X$ with $z \neq w$ and $s\in\mathbb{C}$ with $\Re(s)>2\rho_0$ let
$$
K_{X;\rho_0^2}(z,w;s):=\frac{\Gamma(s-\rho_0)}{\Gamma(s)}\cW_{X,\rho_{0}^{2}}(z, w)(\cosh^{-(s-\rho_{0})}).
$$
Then the function $\Gamma(s)\Gamma^{-1}(s-\rho_0)K_{X;\rho_0^2}(z,w;s)$ admits a meromorphic continuation
to all $s \in \CC$ with poles at points $s=\rho_0 \pm it_j -2m$ for any integer $m\geq 0$. Furthermore, the function
$K_{X;\rho_0^2}(z,w;s)$ satisfies the differential-difference equation
\begin{equation} \label{difference-diff eq for K}
(\Delta_{X}+s(s-2\rho_{0}))K_{X;\rho_0^2}(z,w;s) = s(s+1)K_{X;\rho_0^2}(z,w;s+2).
\end{equation}
\end{cor}

\begin{proof}
By Proposition \ref{p.spectralW}, $K_{X;\rho_0^2}(z,w;s)$ is well defined for $\Re(s)>2\rho_0$.  Keeping
$\Re(s)> 2\rho_0$, we have
\begin{equation} \label{series for K}
K_{X;\rho_0^2}(z,w;s)= \frac{\Gamma(s-\rho_0)}{\Gamma(s)} \sum_{\lambda_j\geq 0} H(t_j, \cosh^{-(s-\rho_{0})}) \psi_j(z)\overline{\psi_j}(w).
\end{equation}
For $\nu\in\mathbb{C}$, $\Re(\nu)>\rho_0$, $r\in\mathbb{R}^+$ or $r\in[0,i\rho_0]$ and non-negative integer $n$, one has that
\begin{equation} \label{H eq for n}
H(r, \cosh^{-(\nu+2n)})= H(r, \cosh^{-\nu})\frac{2^{2n}\left(\frac{\nu+ir}{2}\right)_n\left(\frac{\nu-ir}{2}\right)_n}{(\nu)_{2n}}.
\end{equation}
Indeed, the evaluation of $H(t_j, \cosh^{-(s-\rho_{0})})$ in terms of the Gamma function is stated in (\ref{e.cu}).  One
then can use that $\Gamma(s+1)=s\Gamma(s)$ and the definition of the Pochammer symbol $(s)_{n}$ to arrive at (\ref{H eq for n}).
With this, we can write, for any positive integer $n$,
\begin{equation} \label{K for continuation}
\frac{2^{2n}\Gamma(s)}{\Gamma(s-\rho_0)}K_{X;\rho_0^2}(z,w;s)= \frac{\Gamma(s-\rho_0+2n)}{\Gamma(s-\rho_0)}\sum_{\lambda_j\geq 0}\frac{ H(t_j, \cosh^{-(s-\rho_{0}+2n)})}{Q_n(t_j,s-\rho_0)} \psi_j(z)\overline{\psi_j}(w),
\end{equation}
where
$$
Q_n(r,\nu)=\left(\frac{\nu+ir}{2}\right)_n\left(\frac{\nu-ir}{2}\right)_n.
$$

For $n\geq \lfloor\rho_0\rfloor +1$, the right-hand-side of (\ref{K for continuation}) defines a meromorphic function in the half-plane $\Re(s)>2\rho_0 -2n$
with possible poles at the points $s=\rho_0\pm it_j -2l$, for $l\in{0,...,n-1}$.  Therefore, the function
$\Gamma(s)\Gamma^{-1}(s-\rho_0)K_{X;\rho_0^2}(z,w;s)$
admits a meromorphic continuation to all $\nu \in \CC$ with poles at points $s=\rho_0 \pm it_j -2m$ for integers $m\geq 0$.

It remains to prove the difference-differential equation.
As stated, the right-hand-side of the equation \eqref{series for K} converges absolutely and uniformly on compact
subsets of the right half plane $\Re(s) >  2\rho_0$. When viewed as a function
of $z\in X$, the convergence is uniform on $X$.  Therefore, when restricting $s$ to $\Re(s)>2\rho_0$, we can
interchange the action of $\Delta_X$ and the sum in \eqref{series for K} to get
$$
\Delta_X K_{X;\rho_0^2}(z,w;s)= \frac{\Gamma(s-\rho_0)}{\Gamma(s)} \sum_{\lambda_j\geq 0}(t_j^2 + \rho_0^2) H(t_j, \cosh^{-(s-\rho_{0})}) \psi_j(z)\overline{\psi_j}(w).
$$
Applying \eqref{H eq for n} with $n=1$ we can write the above equation, for sufficiently large $\Re(s)$ as
$$
\Delta_X K_{X;\rho_0^2}(z,w;s)= \frac{\Gamma(s+2-\rho_0)}{\Gamma(s)} \sum_{\lambda_j\geq 0}
\frac{(t_j^2 + \rho_0^2)}{(s-\rho_0)^2 + t_j^2} H(t_j, \cosh^{-(s+2-\rho_{0})}) \psi_j(z)\overline{\psi_j}(w).
$$
Let $n=1$ in \eqref{K for continuation} and multiply by $2^{-2}s(s-2\rho_0)\Gamma(s-\rho_0)\Gamma^{-1}(s)$ to get
$$
s(s-2\rho_0)K_{X;\rho_0^2}(z,w;s) = s(s+1)\frac{\Gamma(s+2-\rho_0)}{\Gamma(s+2)}\sum_{\lambda_j\geq 0}H(t_j, \cosh^{-(s+2-\rho_{0})})
\frac{s^2-2s\rho_0}{(s-\rho_0)^2 + t_j^2}\psi_j(z)\overline{\psi_j}(w).
$$
Adding up the last two equations, we obtained the desired result for sufficiently large $\Re(s)$, and then for all $s$ by meromorphic continuation.
\end{proof}

\begin{rem}\rm
It is necessary to assume that $z \neq w$ when considering the wave distribution
of the test function $g(u) = \cosh^{-(s-\rho_{0})}(u)$.  Only after one computes the spectral expansion of $K_{X;\rho_0^2}(z,w;s)$
is one able to extend the function to $z=w$.
\end{rem}

\section{Two series expansions}

We will define two series using the function $K_{X;\rho_0^2}(z,w;s)$.  The first, in the next Theorem, is shown
to equal the resolvent kernel, which is integral kernel that inverts the operator $(\Delta_X + s(s-2\rho_0))$ for
almost all values of $s$.  As a reminder, the resolvent kernel can be realized as an integral transform of
the heat kernel $K_{X}(z,w;t)$, namely
$$
\int\limits_{0}^{\infty}K_{X}(z,w;t)e^{-s(s-2\rho_{0})t} dt,
$$
provided $z \neq w$ and $\Re(s(s-2\rho_0))>0$.  In that instance, the heat kernel decays exponentially as $t$ approaches zero, so then

\begin{align}\nonumber
\int\limits_{0}^{\infty}K_{X}(z,w;t)e^{-s(s-2\rho_{0})t} dt
&= \lim\limits_{\epsilon \rightarrow 0} \int\limits_{\epsilon}^{\infty}K_{X}(z,w;t)e^{-s(s-2\rho_{0})t} dt
\\&\label{e.resolvent.integral}
= \lim\limits_{\epsilon \rightarrow 0} \sum\limits_{\lambda_{j} \geq 0} \frac{1}{(s-\rho_0)^2 + t_j^2} \psi_j(z)\overline{\psi_j}(w)
\cdot e^{-\epsilon (s(s-\rho_{0})+\lambda_{j})}.
\end{align}

\begin{thm}\label{t.Gseries}
For $z,w\in X$, $z\neq w$ and $s\in\CC$ with $\Re(s)>2\rho_0$ consider the
function
\begin{equation}\label{def.G_series}
G_{X;\rho_0^2}(z,w;s)= \frac{2^{-s-1+\rho_0} \Gamma(s)}{\Gamma(s+1-\rho_0)} \sum_{k=0}^{\infty} \frac{\left( \frac{s}{2}\right) _{k}
\left( \frac{s}{2}+ \frac{1}{2}\right) _{k}}{k! (s+1-\rho_0)_k} K_{X;\rho_0^2}(z,w;s+2k).
\end{equation}
Then we have the following results.
\begin{itemize}
\item[i)] The series defining $G_{X;\rho_0^2}(z,w;s)$ is holomorphic in the half-plane $\Re(s)>2\rho_0$ and continues meromorphically to the whole $s$-plane.

\item[ii)] The function $G_{X;\rho_0^2}(z,w;s)$ admits the spectral expansion
$$
G_{X;\rho_0^2}(z,w;s)= \sum_{\lambda_j\geq 0} \frac{1}{(s-\rho_0)^2 + t_j^2} \psi_j(z)\overline{\psi_j}(w)
$$
which is conditionally convergent, in the sense of (\ref{e.resolvent.integral}), for $z\neq w$
and for all $s \in \CC$ provided $s(s-\rho_{0})+\lambda_{j} \neq 0$ for some $\lambda_{j}$.

\item[iii)] The function $G_{X;\rho_0^2}(z,w;s)$  satisfies the equation
\begin{equation} \label{difference-diff eq for G}
(\Delta_X + s(s-2\rho_0)) G_{X;\rho_0^2}(z,w;s) = 0,
\end{equation}
for all $s \in \CC$ provided $s(s-\rho_{0})+\lambda_{j} \neq 0$ for some $\lambda_{j}$.
\end{itemize}
\end{thm}

\begin{proof}
Let us study each term in the series (\ref{def.G_series}), which is
\begin{align}\nonumber
\frac{2^{-s-1+\rho_0} \Gamma(s)}{\Gamma(s+1-\rho_0)}&\frac{\left( \frac{s}{2}\right) _{k}
\left( \frac{s}{2}+ \frac{1}{2}\right) _{k}}{k! (s+1-\rho_0)_k} K_{X;\rho_0^2}(z,w;s+2k)
\\&= 2^{-1+\rho_{0}}\frac{2^{-(s+2k)} \Gamma(s+2k - \rho_0)}{\Gamma(k+1) \Gamma(s+1-\rho_0+k)}
\cW_{X,\rho_{0}^{2}}(z, w)(\cosh^{-(s+2k-\rho_{0})})\label{e.G_series_term}.
\end{align}

For now, let us assume that $\textrm{\rm Re}(s)>2\rho_0$.
A direct computation using Stirling's formula (\ref{e.Stirling}) yields
that
\begin{equation}\label{g_asymp2}
2^{-1+\rho_{0}}\frac{2^{-(s+2k)} \Gamma(s+2k - \rho_0)}{\Gamma(k+1) \Gamma(s+1-\rho_0+k)}
= O_{s}(k^{-3/2})
\,\,\,\,\,\text{\rm as $k \rightarrow \infty$}.
\end{equation}
It remains to determine the asymptotic behavior of the factor in (\ref{e.G_series_term}) involving the wave distribution. For this,
Theorem \ref{t.wavedistribution} implies that for any $\delta > 0$, there is a $C>1$  depending upon the distance between $z$ and $w$ such that we have the bound
$$
\int\limits_{\delta}^{\infty}F_{z,w}(u) \partial_{u}^{m+1}\left(\cosh^{-(s+2k-\rho_{0})}(u)\right) du = O_{s,z,w}(C^{-(s+2k-\rho_{0})})
\,\,\,\,\,\text{\rm as $k \rightarrow \infty$},
$$
where the implied constant depends on the distance between $z$ and $w$ and $m\geq 4N+1$ is a sufficiently large, fixed integer.
Since $z \neq w$, we can combine equations (\ref{e.ftilde}), (\ref{e.asympPn}), (\ref{e.seventeen}) and (\ref{e.definitionofF})
together with integration by parts to write, for some $C_1>1$
\begin{align}\label{g_asymp3}
\int\limits_{0}^{\delta}F_{z,w}(u) \partial_{u}^{m+1}\left(\cosh^{-(s+2k-\rho_{0})}(u)\right) du &=
(-1)^{m+1}\int\limits_{0}^{\delta}\left(\partial_{u}^{m+1}F_{z,w}(u)\right)\cosh^{-(s+2k-\rho_{0})}(u) du
\\&+
O_{s,z,w}(C_1^{-(s+2k-\rho_{0})})
\,\,\,\,\,\text{\rm as $k \rightarrow \infty$} \notag.
\end{align}
In essence, the use of Theorem \ref{t.wavedistribution} ensures that the boundary term at $u=0$ vanishes, so then
the constant $C_1$ comes from the evaluation of the boundary terms at $u=\delta$ and depends on the distance between $z$ and $w$.  To finish, we may use
(\ref{e.integral_asymp}) where $h(t) = -\log(\cosh(t))$, so then $\lambda = 1$ and $\nu = 2$, to conclude that
\begin{equation}\label{g_asymp4}
\int\limits_{0}^{\delta}\left(\partial_{u}^{m+1}F_{z,w}(u)\right)\cosh^{-(s+2k-\rho_{0})}(u) du =
O_{s,z,w}(k^{-1/2})
\,\,\,\,\,\text{\rm as $k \rightarrow \infty$},
\end{equation}
where, again, the implied constant depends on the distance between $z$ and $w$, and we assume that $z \neq w$.
If we combine (\ref{g_asymp2}), (\ref{g_asymp3}), and (\ref{g_asymp4}),
we obtain that (\ref{e.G_series_term}) is of order $O_{s,z,w}(k^{-2})$.  Therefore, the series
(\ref{def.G_series}) converges uniformly and absolutely for $s$ in compact subsets in a right
half plane $\Re(s)>2\rho_0$, and $z, w \in X$ provided $z$ and $w$ are uniformly bounded apart.

At this point, we have the convergence of the series defining $G_{X;\rho_{0}^{2}}(z,w;s)$ for $\Re(s) > 2\rho_{0}$.
In order to obtain the meromorphic continuation of (\ref{def.G_series}), re-write the series
as a finite sum of terms for $k \leq n$ and an infinite sum for $k > n$, for any integer $n$.  For the finite
sum, the meromorphic continuation is established in Corollary \ref{Lemma K-properties}.  For the infinite
sum, the above argument applies to prove the convergence in the half-plane $\Re(s) > 2\rho_{0} - n$.  With
this, we have completed the proof of assertion (i).

Going further, one can follow the argument given above using (\ref{e.Stirling_derivative}) for any positive integer $\ell$
and conclude that for $z\neq w$, $s$ in some compact subset of the half-plane $\Re(s)>2\rho_0$, we have
\begin{equation*}
\partial_s^{\ell} \left( \frac{2^{-s-1+\rho_0} \Gamma(s)}{\Gamma(s+1-\rho_0)}\frac{\left( \frac{s}{2}\right) _{k}
\left( \frac{s}{2}+ \frac{1}{2}\right) _{k}}{k! (s+1-\rho_0)_k}
K_{X;\rho_0^2}(z,w;s+2k)\right) = O_{s,z,w}(k^{-2-\ell/2})
\,\,\,\,\,\text{\rm as $k \rightarrow \infty$},
\end{equation*}
where the implied constant depends on the distance between $z$ and $w$ and the compact set which contains $s$. Namely,
repeated differentiation of Gamma factors $\ell$ times reduces the exponent by $\ell$, while differentiation of \eqref{g_asymp3}, after application of formula \eqref{e.integral_asymp}, reduces the exponent by $\ell/2$.

The convergence of the series (\ref{def.G_series}) for $\Re(s) > 2\rho_{0}$ as well as the series of
derivatives allows us to interchange differentiation and summation.  Therefore, for $z\neq w$, $\Re(s) > 2\rho_{0}$
and any $\ell\geq 0$ we get
\begin{equation}\label{partial of G}
\partial_s^{\ell} \left( \sum_{k=0}^{\infty}\frac{\left( \frac{s}{2}\right) _{k}\left( \frac{s}{2}+ \frac{1}{2}\right) _{k}}{k! (s+1-\rho_0)_k} K_{X;\rho_0^2}(z,w;s+2k)\right) = \sum_{k=0}^{\infty}\partial_s^{\ell} \left( \frac{\left( \frac{s}{2}\right) _{k}\left( \frac{s}{2}+ \frac{1}{2}\right) _{k}}{k! (s+1-\rho_0)_k} K_{X;\rho_0^2}(z,w;s+2k)\right).
\end{equation}
Now, we would like to include the case $z=w$. Recall that
$$
K_{X;\rho_0^2}(z,w;s)= \frac{\Gamma(s-\rho_0)}{\Gamma(s)}\sum_{\la_j \geq 0}  c_{j, (s-\rho_0)} \psi_j(z)\overline{\psi_j}(w),
$$
where the spectral coefficients $c_{j,(s-\rho_0)}$ are given by \eqref{e.cu}. The coefficients $c_{j,(s-\rho_0)}$
are exponentially decreasing in $t_j= \sqrt{\lambda_j-\rho_0^2}$, as $j\to\infty$ and differentiable with respect to $s$,
with the derivatives also exponentially decreasing in $t_{j}$. Moreover, the application of the the sup-norm bound for the eigenfunctions $\psi_j$ and the Stirling formula for coefficients $c_{j,(s+2k-\rho_0)}$ shows that, uniformly in $z,w\in X$ for $s$ in a compact subset of the half-plane $\Re(s)>2\rho_0$, one has
$$
\left|\cW_{X,\rho_{0}^{2}}(z, w)(\cosh^{-(s+2k-\rho_{0})})  \right| = O_{s,z,w}(k^{N+1}),
$$
where $N$ is the complex dimension of $X$. In addition, repeated differentiation of the coefficients with respect to $s$ reduces the exponent of $k$ by one each time. Therefore, for sufficiently large $\ell$
$$
\sum_{k=0}^{\infty}\left|\partial_s^{\ell} \left( \frac{\left( \frac{s}{2}\right) _{k}\left( \frac{s}{2}+ \frac{1}{2}\right) _{k}}{k! (s+1-\rho_0)_k} K_{X;\rho_0^2}(z,w;s+2k)\right)\right| = O_{s,z,w}(1),
$$
where $\Re(s) > 2\rho_{0}$, and
the bound is uniform in $z,w\in X$. Hence, we may interchange the sum and the integral to get, for sufficiently large $\ell$
\begin{align*}
\int_X \partial_s^{\ell} &\left( \sum_{k=0}^{\infty}\frac{\left( \frac{s}{2}\right) _{k}\left( \frac{s}{2}+ \frac{1}{2}\right) _{k}}{k! (s+1-\rho_0)_k} K_{X;\rho_0^2}(z,w;s+2k)\right)\psi_j(w)\mu(w) \\&= \sum_{k=0}^{\infty}\partial_s^{\ell}\left( \frac{\left( \frac{s}{2}\right) _{k}
\left( \frac{s}{2}+ \frac{1}{2}\right) _{k}}{k! (s+1-\rho_0)_k} \int_X K_{X;\rho_0^2}(z,w;s+2k)\psi_j(w)\mu(w)\right)
\\&=\partial_s^{\ell} \left(\sum_{k=0}^{\infty} \frac{\left( \frac{s}{2}\right) _{k}\left( \frac{s}{2}+ \frac{1}{2}\right) _{k}}{k! (s+1-\rho_0)_k} \int_X K_{X;\rho_0^2}(z,w;s+2k)\psi_j(w)\mu(w)\right),
\end{align*}
where the last equation above follows from the absolute and uniform convergence of the series over $k$, derived in the previous lines.

From the spectral expansion of $K_{X;\rho_0^2}(z,w;s)$ we immediately get
\begin{align*}
\int_X &K_{X;\rho_0^2}(z,w;s+2k)\psi_j(w)\mu(w) \\&=  \frac{2^{s+2k-\rho_0-1}}{\Gamma(s+2k)} \Gamma\left(\frac{s+2k-\rho_0-it_j}{2}\right)\Gamma\left(\frac{s+2k-\rho_0+it_j}{2}\right)\psi_j(z) \\&= \frac{2^{s+2k-\rho_0-1}}{\Gamma(s+2k)} \left(\frac{s-\rho_0-it_j}{2}\right)_k\left(\frac{s-\rho_0+it_j}{2}\right)_k \Gamma\left(\frac{s-\rho_0-it_j}{2}\right)\Gamma\left(\frac{s-\rho_0+it_j}{2}\right)\psi_j(z).
\end{align*}
An application of the doubling formula for the Gamma function yields that
$$
\frac{2^{s+2k-\rho_0-1}\left( \frac{s}{2}\right) _{k}\left( \frac{s}{2}+ \frac{1}{2}\right) _{k}}{\Gamma(s+2k)}=\frac{2^{s-1-\rho_0}}{\Gamma(s)}.
$$
Therefore,
\begin{align*}
\frac{\left( \frac{s}{2}\right) _{k}\left( \frac{s}{2}+ \frac{1}{2}\right) _{k}}{k! (s+1-\rho_0)_k} \int_X &K_{X;\rho_0^2}(z,w;s+2k)\psi_j(w)\mu(w)\\&= \Gamma\left(\frac{s+2k-\rho_0-it_j}{2}\right)\Gamma\left(\frac{s+2k-\rho_0+it_j}{2}\right)\psi_j(z)\\&= \frac{2^{s-1-\rho_0}}{\Gamma(s)} \frac{ \left(\frac{s-\rho_0-it_j}{2}\right)_k\left(\frac{s-\rho_0+it_j}{2}\right)_k }{k! (s+1-\rho_0)_k} \Gamma\left(\frac{s-\rho_0-it_j}{2}\right)\Gamma\left(\frac{s-\rho_0+it_j}{2}\right)\psi_j(z).
\end{align*}
Observe that $\Re\left( \frac{s-\rho_0-it_j}{2} + \frac{s-\rho_0+it_j}{2} - (s+1-\rho_0)\right) =-1 <0$, so then
 the hypergeometric function
$$
\sum_{k=0}^{\infty}\frac{ \left(\frac{s-\rho_0-it_j}{2}\right)_k\left(\frac{s-\rho_0+it_j}{2}\right)_k }{k! (s+1-\rho_0)_k} = F\left( \frac{s-\rho_0-it_j}{2}, \frac{s-\rho_0+it_j}{2}, s+1-\rho_0; 1\right)
$$
is uniformly and absolutely convergent.  From \cite{GR07}, formula 9.122.1 we get
$$
F\left( \frac{s-\rho_0-it_j}{2}, \frac{s-\rho_0+it_j}{2}, s+1-\rho_0; 1\right) = \frac{\Gamma(s+1-\rho_0)}{\Gamma\left(\frac{s-\rho_0+it_j}{2} \right)\Gamma\left(\frac{s-\rho_0-it_j}{2}\right)}\cdot\frac{4}{(s-\rho_0)^2+t_j^2}.
$$
Therefore,
$$
\sum_{k=0}^{\infty}\frac{\left( \frac{s}{2}\right) _{k}\left( \frac{s}{2}+ \frac{1}{2}\right) _{k}}{k! (s+1-\rho_0)_k} \int_X K_{X;\rho_0^2}(z,w;s+2k)\psi_j(w)\mu(w)=\frac{2^{s+1-\rho_0}\Gamma(s+1-\rho_0) }{\Gamma(s) ((s-\rho_0)^2+t_j^2)} \psi_j(z).
$$
This, together with the definition of the function $G_{X;\rho_0^2}(z,w;s)$ yields
\begin{equation}\label{e.Gcomp}
\int_X \partial_s^{\ell} \left( G_{X;\rho_0^2}(z,w;s)\right)\psi_j(w)\mu(w)= \partial_s^{\ell} \left( \frac{1}{(s-\rho_0)^2+t_j^2} \right)\psi_j(z),
\end{equation}
for sufficiently large positive integer $\ell$.

The above computations are valid provided $\Re(s) >2\rho_{0}$.  The arguments could be repeated
with the portion of the series in (\ref{def.G_series}) with $k> n$, for an arbitrary positive integer $n$,  from which one would arrive
at a version of (\ref{e.Gcomp}) where the right-hand-side would have a finite sum of terms subtracted
with the restriction that $\Re(s) > 2\rho_{0}-n$.  However, there is no problem interchanging sum
and differentiation for the finite sum of terms in (\ref{e.Gcomp}) obtained by considering those
with $k \leq n$, from which we conclude that (\ref{e.Gcomp}) holds for all $s$ with $\Re(s) > 2\rho_{0}-n$
provided $s$ is not a pole of (\ref{def.G_series}).

There is a unique meromorphic function $\tilde{G}(z,w;s)$ which is symmetric in $z$ and $w$ and
satisfies $(\Delta_X + s(s-2\rho_0)) \tilde{G}(z,w;s)=0$.  Indeed, for $\Re(s(s-2\rho_0))>0$ one can express $\tilde{G}(z,w;s)$
as an integral transform of the heat kernel, namely
$$
\tilde{G}(z,w;s) = \int\limits_{0}^{\infty}K_X(z, w; t)e^{-s(s-2\rho_{0})t}dt.
$$
At this point, we have that $\tilde{G}(z,w;s) = G_{X;\rho_0^2}(z,w;s) + p_{\ell}(s)$, where
$p_{\ell}(s)$ is a polynomial of degree $\ell$.  The asymptotic behavior as $s$ tends to infinity
can be computed for $G_{X;\rho_{0}^{2}}(z,w;s)$ using Stirling's formula, and that of $\tilde{G}(z,w,s)$
using the above integral expression.  By combining, we get that $p_{\ell}(s) = o(1)$ as $s$ tends
to infinity, thus $p_{\ell}(s) = 0$.

This proves that $G_{X;\rho_0^2}(z,w;s)$ coincides with the conditionally convergent series given as a limit \eqref{e.resolvent.integral} for $\Re(s)>2\rho_0$. Moreover, since both $G_{X;\rho_0^2}(z,w;s)$ and the resolvent kernel $\tilde{G}(z,w;s)$ possess meromorphic continuation to the whole complex $\CC-$plane, they must coincide.

With all this, assertions (ii) and (iii) are established.
\end{proof}

\begin{thm}\label{E_series}
Let
$$
E_{X;\rho_{0}^{2}}(z,w;s) = \frac{\Gamma((s+1-2\rho_{0})/2)}{\Gamma(s/2)}\sum\limits_{k=0}^{\infty}
\frac{\left( \frac{s}{2}\right) _{k}}{k!}K_{X;\rho_0^2}(z,w;s+2k)
$$
for $z,w\in X$, $z\neq w$. Then $E_{X;\rho_{0}^{2}}(z,w;s)$ converges to a meromorphic function for $\textrm{\rm Re}(s) < 0$
away from the poles of any $K_{X;\rho_0^2}(z,w;s+2k)$ and negative integers.  Furthermore,
$E_{X;\rho_{0}^{2}}(z,w;s)$ extends to a meromorphic function for all $s$ and
satisfies the differential-difference equation
\begin{equation} \label{difference-diff eq for E}
(\Delta_X + s(s-2\rho_0)) E_{X;\rho_{0}^{2}}(z,w;s) = -s^{2}E_{X;\rho_{0}^{2}}(z,w;s+2)
\end{equation}
\end{thm}

\begin{proof}
The estimates in the proof of Theorem \ref{t.Gseries}, namely
(\ref{g_asymp2}), (\ref{g_asymp3}), and (\ref{g_asymp4}) combine to show that
$$
\frac{\left( \frac{s}{2}\right) _{k}}{k!}K_{X;\rho_0^2}(z,w;s+2k)= O_{s,z,w}(k^{s/2-\rho_{0}-3/2})
\,\,\,\,\,\text{\rm as $k \rightarrow \infty$},
$$
where the implied constant depends upon $s$ and upon the distance between points $z$ and $w$. Therefore, the series converges for $s$ with $\textrm{Re}(s)<0$ provided no term has a pole.  Set
$$
\tilde{E}_{X;\rho_{0}^{2}}(z,w;s) = \sum\limits_{k=0}^{\infty}
\frac{\left( \frac{s}{2}\right) _{k}}{k!}K_{X;\rho_0^2}(z,w;s+2k).
$$
Using the difference-differential equation for $K_{X;\rho_0^2}(z,w;s)$, as established in Corollary \ref{Lemma K-properties},
we can prove such an equation for $\tilde{E}_{X;\rho_{0}^{2}}(z,w;s)$.  Indeed, for $\Re(s)\ll 0$  begin by writing
\begin{align*}
(\Delta_{X} &+ s(s-2\rho_{0}))\tilde{E}_{X;\rho_{0}^{2}}(z,w;s) =
\sum\limits_{k=0}^{\infty}\left(\frac{\left( \frac{s}{2}\right) _{k}}{k!}\Delta_{X}K_{X;\rho_0^2}(z,w;s+2k) + \frac{\left( \frac{s}{2}\right) _{k}}{k!}s(s-2\rho_{0})
K_{X;\rho_0^2}(z,w;s+2k)\right) \\
&=\sum\limits_{k=0}^{\infty}\frac{\left( \frac{s}{2}\right) _{k}}{k!}\left(-(s+2k)(s+2k-2\rho_{0})K_{X;\rho_0^2}(z,w;s+2k)+ (s+2k)(s+2k+1)K_{X;\rho_0^2}(z,w;s+2k+2)\right)\\
&\hspace{5mm} + \sum\limits_{k=0}^{\infty} \frac{\left( \frac{s}{2}\right) _{k}}{k!}s(s-2\rho_{0})K_{X;\rho_0^2}(z,w;s+2k)\\
&= \sum\limits_{k=1}^{\infty} \frac{\left( \frac{s}{2}\right) _{k}}{k!} \left(-(s+2k)(s+2k-2\rho_{0}) + s(s-2\rho_{0})\right)K_{X;\rho_0^2}(z,w;s+2k) \\
&\hspace{5mm} + \sum\limits_{k=0}^{\infty}\frac{\left( \frac{s}{2}\right) _{k}}{k!}(s+2k)(s+2k+1)K_{X;\rho_0^2}(z,w;s+2k+2)\\
&= \sum\limits_{n=0}^{\infty} \frac{\left( \frac{s}{2}\right) _{n+1}}{(n+1)!} \left(-(s+2n+2)(s+2n+2-2\rho_{0}) + s(s-2\rho_{0})\right)K_{X;\rho_0^2}(z,w;s+2n+2) \\
&\hspace{5mm} + \sum\limits_{n=0}^{\infty}\frac{\left( \frac{s}{2}\right) _{n}}{n!}(s+2n)(s+2n+1)K_{X;\rho_0^2}(z,w;s+2n+2).
\end{align*}
Since
$$
-(s+2n+2)(s+2n+2-2\rho_{0}) + s(s-2\rho_{0}) = -(2n+2)(2s+2n+2-2\rho_{0}),
$$
the coefficient of $K_{X;\rho_0^2}(z,w;s+2n+2)$ in the last expression is
$$
-\frac{\left( \frac{s}{2}\right) _{n+1}}{(n+1)!}(2n+2)(2s+2n+2-2\rho_{0})
+ \frac{\left( \frac{s}{2}\right) _{n}}{n!}(s+2n)(s+2n+1).
$$
Using the definition of the Pochhammer symbol, it is elementary to show that
$$
-\frac{\left( \frac{s}{2}\right) _{n+1}}{(n+1)!}(2n+2)(2s+2n+2-2\rho_{0})
+ \frac{\left( \frac{s}{2}\right) _{n}}{n!}(s+2n)(s+2n+1)
= \frac{\left( \frac{s+2}{2}\right) _{n}}{n!}(-s(s+1-2\rho_{0})),
$$
hence we arrive at the equation
$$
(\Delta_X + s(s-2\rho_0)) \tilde{E}_{X;\rho_{0}^{2}}(z,w;s) = -s(s+1-2\rho_{0})\tilde{E}_{X;\rho_{0}^{2}}(z,w;s+2).
$$

Notice that
$$
E_{X;\rho_{0}^{2}}(z,w;s)= \frac{\Gamma((s+1-2\rho_{0})/2)}{\Gamma(s/2)}
\tilde{E}_{X;\rho_{0}^{2}}(z,w;s),
$$
so then
\begin{align*}
(\Delta_X + s(s-2\rho_0)) E_{X;\rho_{0}^{2}}(z,w;s) &= \frac{\Gamma((s+1-2\rho_{0})/2)}{\Gamma(s/2)}\left(-s(s+1-2\rho_{0})\right)
\tilde{E}_{X;\rho_{0}^{2}}(z,w;s+2)
\\&= -s^{2}
\frac{\Gamma((s+1-2\rho_{0})/2)}{\Gamma(s/2)}\frac{(s+1-2\rho_{0})/2}{s/2}
\tilde{E}_{X;\rho_{0}^{2}}(z,w;s+2)
\\&= -s^{2}
\frac{\Gamma(((s+2)+1-2\rho_{0})/2)}{\Gamma((s+2)/2)} \tilde{E}_{X;\rho_{0}^{2}}(z,w;s+2)
\\& -s^{2} E_{X;\rho_{0}^{2}}(z,w;s+2),
\end{align*}
as asserted.
\end{proof}

\begin{rem}\rm
The motivation of the series in Theorem \ref{E_series} is the following elementary formula first employed in the context of
elliptic Eisenstein series in \cite{vP10}.  For any $x$ with $\vert x \vert < 1$ and complex $s$, one has
the convergent Taylor series
$$
(1-x)^{-s/2} = \sum\limits_{k=0}^{\infty}\frac{\left(\frac{s}{2}\right)_{k}}{k!}x^{k}.
$$
By setting $x = (\cosh u)^{-2}$, one then gets that
$$
(1-(\cosh u)^{-2})^{-s/2} = \sum\limits_{k=0}^{\infty}\frac{\left(\frac{s}{2}\right)_{k}}{k!}(\cosh u)^{-2k}.
$$
Now write
$$
(1-(\cosh u)^{-2})^{-s/2}  =(\cosh u)^{s}((\cosh u)^{2}-1)^{-s/2}= (\cosh u)^{s}(\sinh u)^{-s}
$$
from which we obtain the identity
$$
\sinh^{-s}(u) = \sum\limits_{k=0}^{\infty}\frac{\left(\frac{s}{2}\right)_{k}}{k!}\cosh^{-(s+2k)}(u).
$$
In this way, we can study the function obtained by applying the wave distribution to $g(u) = \sinh^{-s}(u)$,
even though this function does not satisfy the conditions of Theorem \ref{t.wavedistribution}.  Indeed, this observation
is the motivation behind the definition of $E_{X;\rho_{0}^{2}}(z,w;s)$.
\end{rem}

\section{Kronecker limit formulas}

We now prove the Kronecker limit formulas for $G_{X;\rho_0^2}(z,w;s)$ and $E_{X;\rho_0^2}(z,w;s)$,
meaning we analyze the first two terms in the Laurent series at $s=0$.  We will continue assuming
that $\rho_{0} \geq 0$ is arbitrary.   The choice of $\rho_{0}$ plays no role in the analysis
of $G_{X;\rho_0^2}(z,w;s)$.  However, in the approach taken in this section, the case when
$\rho_{0}=1/2$ will be particularly interesting when studying $E_{X;\rho_0^2}(z,w;s)$.

\begin{cor}\label{c.Gasymp}  If $\rho_{0} \neq 0$, then the function $G_{X;\rho_0^2}(z,w;s)$ has the asymptotic behavior
$$
G_{X;\rho_0^2}(z,w;s) = \frac{-1/(2\rho_{0})}{\textrm{\rm vol}_{\omega}(X)}s^{-1} + G_{X;\rho_0^2}(z,w) + \frac{1/(2\rho_{0})^{2}}{\textrm{\rm vol}_\omega (X)}+ O(s)
\,\,\,\,\,
\textrm{as $s \rightarrow 0$}
$$
where $G_{X;\rho_0^2}(z,w)$ is the Green's function associated to the Laplacian $\Delta_X$
acting on the space of smooth functions on $X$ which are orthogonal to the constant functions.
If $\rho_{0}=0$, then we have the expansion
$$
G_{X;\rho_0^2}(z,w;s) = \frac{1}{\textrm{\rm vol}_\omega (X)}s^{-2} + G_{X;\rho_0^2}(z,w) + O(s)
\,\,\,\,\,
\textrm{as $s \rightarrow 0$}
$$
\end{cor}

\begin{proof}
The result follows directly from part (ii) of Theorem \ref{t.Gseries} having noted that the eigenfunction
associated to the zero eigenvalue is $1/(\textrm{\rm vol}_\omega (X))^{1/2}$ and that
$$
\frac{1}{s(s-2\rho_{0})} = \frac{-1/(2\rho_{0})}{s} + \frac{-1/2\rho_{0}}{s-2\rho_{0}} = \frac{-1/(2\rho_{0})}{s} + \frac{1}{(2\rho_{0})^{2}}
+ O(s)
\,\,\,\,\,\textrm{\rm as $s \rightarrow 0$}
$$
in the case $\rho_{0} \neq 0$.  If $\rho_{0}=0$, the assertion follows immediately from part (ii) of Theorem \ref{t.Gseries}.
\end{proof}

\begin{rem}\rm
Corollary \ref{c.Gasymp} is, in some sense, elementary and well-known.  Indeed, using (\ref{e.resolvent.integral}), we can write
$$
G_{X;\rho_0^2}(z,w;s) = \int\limits_{0}^{\infty}\left(K_{X}(z,w;t)e^{-s(s-2\rho_{0})t}- \frac{1}{\textrm{\rm vol}_\omega (X)}\right)dt
+  \frac{1}{s(s-2\rho_{0})}\frac{1}{\textrm{\rm vol}_\omega (X)}.
$$
As $s$ approaches zero while $\Re(s(s-2\rho_0))>0$, the above integral converges to the Green's function $G_{X;\rho_0^2}(z,w)$.
Nonetheless, the novel aspect of Theorem \ref{t.Gseries} is the expression of the resolvent kernel as a series.
\end{rem}

\begin{rem}\rm
The statement of Corollary \ref{c.Gasymp} highlights the difference between the cases when $\rho_{0}=0$ and $\rho_{0}\neq 0$.
The difference determines the order of the singularity of the resolvent kernel $G_{X;\rho_0^2}(z,w;s)$ at $s=0$.  Of course,
one could re-write Corollary \ref{c.Gasymp} as
$$
G_{X;\rho_0^2}(z,w;s) = \frac{1}{\textrm{\rm vol}_\omega (X)}\frac{1}{s(s-2\rho_{0})} + G_{X;\rho_0^2}(z,w) + O(s)
\,\,\,\,\,
\textrm{as $s \rightarrow 0$,}
$$
which includes both $\rho_{0}=0$ and $\rho_{0}\neq 0$.
\end{rem}

\begin{thm}\label{t.G_KLF}
Let $D$ be the divisor of a holomorphic form $F_{D}$ on $X$, and assume that $D$ is smooth up to codimension two in $X$.
Then, for $z\notin D$, there exist constants $c_{0}$ and $c_{1}$ such that
\begin{equation}\label{e.Gseries_KLF}
\int_{D}G_{X;\rho_0^2}(z,w;s)\mu_{D}(w) = \frac{\textrm{\rm vol}_\omega (D)}{\textrm{\rm vol}_\omega (X)}\frac{1}{s(s-2\rho_{0})} +
c_{0}\log \Vert F_{D}(z)\Vert^{2}_{\omega} +c_{1}+ O(s)
\,\,\,\,\,
\textrm{as $s \rightarrow 0$.}
\end{equation}
\end{thm}

\begin{proof}
For now, assume that $z\notin D$ and $w \in D$.
By part (i) of Theorem \ref{t.Gseries}, the function $G_{X;\rho_0^2}(z,w;s)$ is holomorphic in
$s$ for $\Re(s) > 2\rho_{0}$, so then the integral in (\ref{e.Gseries_KLF}) exists for $\Re(s) > 2\rho_{0}$.
The integral has a meromorphic continuation in $s$, again by part (i) of Theorem \ref{t.Gseries}, and
the Laurent expansion of the integral near $s=0$ can be evaluated integrating over $D$ the
expansion given in Corollary \ref{c.Gasymp}.
The singularity of the Green's function $G_{X;\rho_0^2}(z,w)$ as $z$ approaches $w$ is known;
see, for example, page 94 of \cite{Fo76} as well as \cite{JK98} and \cite{JK01}.  In the latter references, the authors carefully evaluate
the integrals of functions with Green's function type singularities; see section 3 of \cite{JK98}.  From those arguments, we conclude that
$$
\int_{D}G_{X;\rho_0^2}(z,w;s)\mu_{D}(w)
$$
has a logarithmic singularity as $z$ approaches $D$.

Throughout this discussion the Laplacian $\Delta_{X}$ acts on the variable $z$.  From the equation
$$
(\Delta_X + s(s-2\rho_0)) G_{X;\rho_0^2}(z,w;s) = 0,
$$
as proved in Theorem \ref{t.Gseries}, and the expansion in Corollary \ref{c.Gasymp},
we conclude that for $z \neq w$, we have
\begin{equation}\label{e.Laplacian.G}
\Delta_X G_{X;\rho_0^2}(z,w) = \frac{2}{\textrm{\rm vol}_\omega (X)}.
\end{equation}
Let us consider the difference
\begin{equation}\label{e.difference}
\int_{D}G_{X;\rho_0^2}(z,w;s)\mu_{D}(w) -\frac{\textrm{\rm vol}_\omega (D)}{\textrm{\rm vol}_\omega (X)}\frac{1}{s(s-2\rho_{0})}-
c_{0}\log \Vert F_{D}(z)\Vert^{2}_{\omega}
\end{equation}
near $s=0$.  For any $c_{0}$, the difference is holomorphic in $s$ near $s=0$.
From section \ref{s.forms}, we have that $\Delta_X \log \Vert F_{D}(z)\Vert^{2}_{\omega}$ is a non-zero constant.  Choose $c_{0}$ so
that
$$
c_{0} \Delta_X \log \Vert F_{D}(z)\Vert^{2}_{\omega} = \frac{2}{\textrm{\rm vol}_\omega (X)}.
$$
By combining (\ref{e.Laplacian.G}) and (\ref{e.omega.Laplacian}), we conclude that
$$
\textrm{\rm d}\textrm{\rm d}^{c} \int_{D}G_{X;\rho_0^2}(z,w)\mu_{D}(w) = \delta_{D'} - \omega
$$
where $D'$ is a divisor whose support is equal to the support of $D$.  It remains to show that
$D'=D$.

Consider the difference
\begin{equation}\label{e.difference.KLF}
R_{X;\rho_{0}^{2}}(z;D):= \int_{D}G_{X;\rho_0^2}(z,w)\mu_{D}(w) - \frac{c_{0}}{\textrm{\rm vol}_\omega (X)}\log \Vert F_{D}(z)\Vert^{2}_{\omega}.
\end{equation}
which satisfies
$$
\textrm{\rm d}\textrm{\rm d}^{c} R_{X;\rho_{0}^{2}}(z;D) = \delta_{D'} - n\delta_{D},
$$
which means that $R_{X;\rho_{0}^{2}}(z;D)$ is harmonic away from the support of $D$ and has logarithmic
growth as $z$ approaches $D$.  If $X$ is an algebraic curve, then $D$ is a finite sum of points, say
$D = \sum m_{j}D_{j}$, with multiplicities $m_{j}$.  In this case,
$$
\int_{D}G_{X;\rho_0^2}(z,w)\mu_{D}(w) = \sum m_{j} G_{X;\rho_0^2}(z,D_{j}).
$$
It follows that $D'=D$.  By the Riemann removable singularity theorem, the difference (\ref{e.difference.KLF})
is harmonic on all of $X$, hence bounded, which implies that $R_{X;\rho_{0}^{2}}(z;D)$ is a constant.
The argument for general $X$ is only slightly different.  Again, write $D = \sum m_{j}D_{j}$ where
each $D_{j}$ is  irreducible.  Choose a smooth point $P$ on $D$, hence on some $D_{j}$.
One can express the integral in
(\ref{e.difference.KLF}) near $P$ using suitably chosen local coordinates in $X$, as in section 3 of \cite{JK98}.
By doing so, one again concludes that the value of the integral of $G_{X;\rho_0^2}(z,w)$ as $P$ approaches
$D$ is equal to the coefficient of $D_{j}$.  Therefore, the difference (\ref{e.difference}) is bounded as $z$
approaches $P$.  Since $D$ is smooth in codimension two, we can again apply the Riemann removable singularity theorem
(see Corollary 7.3.2, page 262 of \cite{Kr82}) to conclude that $R_{X;\rho_{0}^{2}}(z;D)$ is bounded and harmonic
on $X$, hence constant.
\end{proof}

\begin{rem}\rm
The constant $c_{0}$ can be expressed as a function of the weight of the form $F_D$.
The constant $c_{1}$ of (\ref{e.Gseries_KLF}) can be determined by integrating both sides of (\ref{e.Gseries_KLF})
with respect to $z$, using that the integral of the Green's function $G_{X;\rho^{2}_{0}}(z,w)$ is zero.  This
will express $c_{1}$ as an integral of $\log \Vert F_{D}(z)\Vert^{2}_{\omega}$.
\end{rem}

\begin{rem}\rm
In effect, the proof of Theorem \ref{t.G_KLF} requires that the norm $\Vert F_{D}\Vert_{\omega}$ is such that
the Laplacian $\Delta_{X}$ of its logarithm is constant, so then a certain linear combination of the Green's function and
$\log \Vert F_{D}\Vert_{\omega}$ has Laplacian equal to zero away from $D$.  This statement can hold in settings not
covered by the conditions stated of Theorem \ref{t.G_KLF}.   In this setting, the forms $F_{D}$ one studies
are determined by the condition that $\log \vert F_{D} \vert $ are harmonic even in the setting when a complex
structure does not exist.  In fact, this requirement is true
for quotients of hyperbolic $n$-spaces  of any dimension $n \geq 2$.
\end{rem}

\begin{rem}\rm
Suppose we are given a codimension one subvariety $D$ of $X$, and assume that $D$ is smooth in codimension
one.  Then one can realize the log norm of the form $F_{D}$ which vanishes along $D$ by (\ref{e.Gseries_KLF}).  In this
manner, we can construct $F_{D}$ when $D$ has been given.  The form $F_{D}$ need not be a holomorphic form,
but rather a section of the canonical bundle twisted by a flat line bundle.  The parameters of the flat
line bundle can be viewed as a generalization of Dedekind sums since classical Dedekind sums stem from
attempting to drop the absolute values from the Kronecker limit function associated to the parabolic Eisenstein
series for $\textrm{\rm PSL}(2,\ZZ)$.
\end{rem}

Let us now extend the development of Kronecker limit functions to $E_{X;\rho_0^2}(z,w;s)$.  To do so,
we first proof that for certain $\rho_{0}^{2}$, the functions $G_{X;\rho_0^2}(z,w;s)$ and $E_{X;\rho_0^2}(z,w;s)$
have the same expansion at $s=0$ out to $O(s^{2})$.

\begin{prop}\label{p.GEdifference} For $z,w \in X$, $z\neq w$ consider the difference
$$
D_{X;\rho_0^2}(z,w;s):=E_{X;\rho_0^2}(z,w;s) - 2^{s+1-\rho_{0}}\frac{\Gamma(s+1-\rho_{0})\Gamma((s+1-2\rho_{0})/2)}
{\Gamma(s)\Gamma(s/2)}G_{X;\rho_0^2}(z,w;s).
$$
Then for all $\rho_{0}\geq 0$, such that $\rho_0\neq m$ or $\rho_0\neq m+1/2,$ for integers $m\geq 1$ we have that $D_{X;\rho_0^2}(z,w;s) = O(s^{2})$ as $s \rightarrow 0$.
\end{prop}

\begin{proof}
The factor
\begin{equation}\label{G factor}
2^{s+1-\rho_{0}}\frac{\Gamma(s+1-\rho_{0})\Gamma((s+1-2\rho_{0})/2)}
{\Gamma(s)\Gamma(s/2)}
\end{equation}
of $G_{X;\rho_0^2}(z,w;s)$ was chosen so that the $k=0$ term in the series
expansions for $E_{X;\rho_0^2}(z,w;s)$ and $G_{X;\rho_0^2}(z,w;s)$ agree.  For
any $k \geq 1$, the $k$-th term in the series expansion for the difference $D_{X;\rho_0^2}(z,w;s)$
is
\begin{equation}\label{e.k_th_term}
\frac{\Gamma((s+1-2\rho_{0})/2)}{\Gamma(s/2)}\frac{\left( \frac{s}{2}\right) _{k}}{k!}
\left( 1-  \frac{\left( \frac{s}{2}+ \frac{1}{2}\right) _{k}}{(s+1-\rho_0)_k} \right)K_{X;\rho_0^2}(z,w;s+2k)
\end{equation}
From the spectral expansion in Proposition \ref{p.spectralW}, the function $K_{X;\rho_{0}}(z,w;s+2k)$
is holomorphic at $s=0$ for any integer $k\geq 1$. When $\rho_0$ is distinct from any positive half-integer, the function $\Gamma((s+1-2\rho_{0})/2)$ is also holomorphic at $s=0$, while $\frac{\left( \frac{s}{2}+ \frac{1}{2}\right) _{k}}{(s+1-\rho_0)_k}$ is holomorphic at $s=0$ for $\rho_0$ distinct from any positive integer. Since there is a factor of $\Gamma^{-2}(s/2)$ in
the above coefficient, it follows that the function $D_{X;0}(z,w;s)$ is $O(s^{2})$ as $s$ approaches zero, for all $\rho_0$ different from positive integers or half-integers.

It remains to prove the statement for $\rho_{0}=1/2$, in which case (\ref{e.k_th_term}) becomes
$$
\frac{\left( \frac{s}{2}\right) _{k}}{k!}
\left( 1-  \frac{\left( \frac{s}{2}+ \frac{1}{2}\right) _{k}}{(s+1/2)_k} \right)K_{X;1/4}(z,w;s+2k)
= \frac{\left( \frac{s}{2}\right) _{k}}{k!}
\left(\frac{(s+1/2)_k -
\left( \frac{s}{2}+ \frac{1}{2}\right) _{k}}{(s+1/2)_k} \right)K_{X;1/4}(z,w;s+2k)
$$
For $k \geq 1$, the factor $(s/2)_{k}$ vanishes at $s=0$, as does the difference $(s+1/2)_{k} - (s/2+1/2)_{k}$,
so it follows that the function $D_{X;1/4}(z,w;s)$ is $O(s^{2})$ as $s$ approaches zero.
\end{proof}

\begin{cor}\label{c.Easymp}  For $z,w\in X$, $z\neq w$ and $\rho_0> 0$ which is not equal to a positive integer or half-integer, $E_{X;\rho_0^2}(z,w;s)=O(s)$, as $s \to 0$. When $\rho_0=1/2$, there are constants $b_{0}$, $b_{1}$ and $b_2$ such that the
function $E_{X;1/4}(z,w;s)$ has the asymptotic behavior
$$
E_{X;1/4}(z,w;s) = b_{0} + (b_{1} + b_2 G_{X;1/4}(z,w))s + O(s^{2})
\,\,\,\,\,
\textrm{as $s \rightarrow 0$}
$$
where $G_{X;\rho_0^2}(z,w)$ is the Green's function associated to the Laplacian $\Delta_X$
acting on the space of smooth functions on $X$ which are orthogonal to the constant functions.
\end{cor}

\begin{proof} When $\rho_0> 0$ is not equal to a positive integer or half-integer,  the functions $\Gamma(s+1-\rho_{0})$ and $\Gamma((s+1-2\rho_{0})/2)$ are holomorphic at $s=0$, hence the factor \eqref{G factor} is $O(s^2)$ as $s$ approaches zero. Combining this with Proposition \ref{p.GEdifference} and Corollary \ref{c.Gasymp} yields the statement.

When $\rho_0= 1/2$
$$
2^{s+1-\rho_{0}}\frac{\Gamma(s+1-\rho_{0})\Gamma((s+1-2\rho_{0})/2)}
{\Gamma(s)\Gamma(s/2)} =
2^{s+1/2}\frac{\Gamma(s+1/2)}{\Gamma(s)}= a_1 s + a_2 s^2 + O(s^3)
$$
so the statement of Proposition \ref{p.GEdifference} becomes
$$
E_{X;1/4}(z,w;s)= (a_1s + a_2s^2 + O(s^3))\left( \frac{1}{\textrm{\rm vol}_\omega (X)}s^{-1} + G_{X;1/4}(z,w) + O(s) \right) + O(s^2),
$$
as $s\to 0$. Multiplying the above expression we deduce the statement.
\end{proof}

\begin{rem}\rm
When $\rho_0=0$ the term in \eqref{G factor} is $\sqrt{\pi} s^2 + O(s^3)$ as $s \to 0$. When combined with Proposition \ref{p.GEdifference}
and Corollary \ref{c.Gasymp}, one gets that
$$
E_{X;0}(z,w;s)=\sqrt{\pi}/\mathrm{Vol}_\omega (X)+ O(s^2)
\,\,\,\,\,
\textrm{\rm as $s \rightarrow 0$.}
$$
\end{rem}

\begin{rem}\rm
In the case when $X$ is a hyperbolic Riemann surface, the result of Proposition \ref{p.GEdifference}
is stated in Corollary 7.4 of \cite{vP16}, with a slightly different renormalization constant $\sqrt{2\pi}$ in front of the Green's function, which stems from a different constant term in the definition of the corresponding series.  However, in their proof, the author used special
function identities which are specific to that setting.
\end{rem}

\begin{rem}\rm \label{r.E_KLF}
The constants $b_{0}$, $b_{1}$ and $b_2$ in case when $\rho_0= 1/2$ are easily evaluated using asymptotic behavior of the factor \eqref{G factor} near $s=0$, which are not so significant to us at this point.
What does matter is that for $\rho_0=1/2$, we have that $E_{X;1/4}(z,w;s)$ admits a Kronecker limit formula.  In the notation of Theorem \ref{t.G_KLF}, there are constants $c_{0}$, $c_{1}$ and
$c_{2}$ such that
$$
\int_{D}E_{X;1/4}(z,w;s)d\mu_{D}(w) = c_{0}\textrm{\rm vol}_\omega (D) + \left(c_{1}\log \Vert F_{D}(z)\Vert^{2}_{\mu} +c_{2}\right)s+ O(s^{2})
\,\,\,\,\,
\textrm{as $s \rightarrow 0$}.
$$
\end{rem}

\begin{rem}\rm
It is important to note that we have not excluded the possibility of a ``nice'' Kronecker limit function for
$E_{X;\rho_{0}^{2}}(z,w;s)$ when $\rho_{0} \neq 1/2$.  The approach we took in this article was to compare
the Kronecker limit function of $E_{X;\rho_{0}^{2}}(z,w;s)$ to that of the resolvent kernel $G_{X;\rho_{0}^{2}}(z,w;s)$.
We find it quite interesting that the comparison yields a determination of the Kronecker limit function of
$E_{X;\rho_{0}^{2}}(z,w;s)$ only in the case when $\rho_{0}=1/2$.
\end{rem}

\begin{rem}\rm
Ultimately, we are interested in the cases when $X$ is the quotient of a symmetric space $G/K$. In this
setting, $\rho_{0}$ is zero only when $G/K$ is Euclidean.  In all other cases, $\rho_{0}^2$ is positive. (See section 1.3.) 
\end{rem}

\section{Examples}
\label{s.Examples}

As stated above, we began our analysis with the heat kernel and obtained our results using its
spectral expansion.  As one could imagine, any other representation of the heat kernel has the potential
of combining with our results to yield formulas of possible interest.  We will proceed along these lines
and introduce three examples.  It is our opinion that each example is of independent interest.  Rather
than expanding upon any one example, we will present, in rather broad strokes, the type of formulas which
will result, and we will leave a detailed analysis for future work.

\subsection{Abelian varieties}
Let $\Omega$ be an $N\times N$ complex matrix which is symmetric and whose imaginary part is postive
definite.  Let $\Lambda_{\Omega}$ denote the $\ZZ$-lattice formed
by vectors in $\ZZ^{N}$ and $\Omega \ZZ^{N}$.  Let $X$ be an abelian variety whose
complex points form the $N$-dimension complex torus $\CC^{N}/(\ZZ^{N} \otimes \Omega \ZZ^{N})$.  Assume
that $X$ is equipped with its natural flat metric induced from the Euclidean metric on $\CC^{N}$.  It
can be shown that all eigenfunctions on the associated Laplacian are exponential functions.  In addition,
the heat kernel on $X$ can be obtained by periodizing over $\Lambda_{\Omega}$ the heat kernel on $\CC^{N}$.
By the uniqueness of the heat kernel on $X$, one obtains a formula of the type
$$
K_X(z,w; t) = \sum_{k = 0}^\infty e^{-\la_k t} \psi_k(z) \psi_k(w) = \sum\limits_{v \in \Lambda_{\Omega}}
\frac{1}{(4\pi t)^{N}}e^{-\Vert z - w - v\Vert^{2}/(4t)}
$$
where $\Vert \cdot \Vert$ denotes the absolute value in $\CC^{N}$.  In effect, the identity obtained by
equating the above two expressions for the heat kernel is the Poisson summation formula.  In the setting of section 3, we take
$\rho_{0}^{2}=0$, so then the Poisson kernel (\ref{e.DefPoissonKernel}) becomes
\begin{equation}\label{e.P_torus}
\cP_{X, 0}(z, w; u) = \frac{u}{\sqrt{4\pi}} \int_0^\infty
K_X(z, w; t) e^{-u^2/(4t)} t^{-1/2} \, \frac{dt}{t}
= \sum\limits_{v \in \Lambda_{\Omega}}\frac{u \Gamma(N+1/2)}{\pi(u^{2} + \Vert z - w - v\Vert^{2})^{N+1/2}}.
\end{equation}

As is evident, one cannot simply replace $u$ by $iu$ in (\ref{e.P_torus}) since then the sum would have singularities
whenever $u^{2} = \Vert z - w - v\Vert^{2}$.  However, this is where the distribution theory approach is necessary and,
indeed, one will obtain the function $K_{X;0}(z,w;s)$.  For now, one can formally express $K_{X;0}(z,w;s)$ as the
integral of $\cosh^{-s}(u)$.  In the notation of Theorem \ref{t.G_KLF}, one can take $D$ to be the theta divisor of the
Riemann theta function $\Theta$ on $X$.  The Kronecker limit formula for $\log \Vert \Theta \Vert$ then could be
viewed as coming from the series over $\Lambda_{\Omega}$.  Upon exponentiation, one would have a product formula, or
regularized product, formula for $\Vert \Theta \Vert^{2}$.  Certainly, the exploration of this example is worthy of study.

\subsection{Complex projective space}
Let $\om_{FS}$ denote the Fubini-Study metric on complex projective space $\CC\PP^n$.  The authors
in \cite{HI02} derived an explicit expression for the heat kernel $K_{\CC\PP^n}$ associated to
the Laplacian of the Fubini-Study metric on $\CC\PP^n$.  Specifically, it is proved that
\beeq
\label{e.HeatinCPn}
K_{\CC\PP^n} (z, w; t)= \frac{e^{n^2t}}{2^{n-2}\pi^{n+1}} \int_r^{\pi/2} \frac{-d(\cos u)}{\sqrt{\cos^2 r - \cos^2u}}
\left( -\frac{1}{\sin u} \frac{d}{du} \right)^n [\Theta_{n+1}(t,u)],
\eeeq
where $z, w \in \CC\PP^n$, $t > 0$, and $r  = \textrm{\rm dist}_{g_{FS}}(z, w) = \tan^{-1}(|z - w|)$, and the function $\Theta_{n+1}(t,u)$
is given by
$$
\Theta_{n+1}(t,u) = \sum_{\ell = 0}^\infty e^{-4t(\ell + n/2)^2} \cos((2\ell + n)u).
$$
Equivalently, one can write
\beeq
\label{e.goodheat}
K_{\CC\PP^n} (z, w, t)= \sum_{\ell = 0}^\infty e^{-\la_\ell t}
 \theta_\ell(r),
\eeeq
where $\la_\ell  = 4\ell(\ell + n)$, and
\beeq
\label{e.thetaell}\notag
\theta_\ell(r) =  \frac{1}{2^{n-2}\pi^{n+1}}\int_r^{\pi/2} \frac{\sin \tau}{\sqrt{\cos^2 r - \cos^2\tau}}
\left( -\frac{1}{\sin \tau} \frac{d}{d\tau} \right)^n \cos((2\ell + n)\tau) \, d\tau.
\eeeq

As in the previous example, the formula for the heat kernel is explicit, and all
integral transforms leading up to the resolvent kernel $G_{X;\rho^{2}}(z,w;s)$ and
$E_{X;\rho^{2}}(z,w;s)$ can be evaluated, at least formally.  It seems as if one
would also take $\rho_{0}^{2}=0$ in this case, though it would be worthwhile to
consider $\rho_{0}^{2}=1/2$ as well.  Of course, the divisors to consider would be the zeros
of homogenous polynomials in $N$-variables, and the norm of homogenous polynomials would be
with respect to the Fubini-Study metric.

\subsection{Compact quotients of symmetric spaces}

Let $G$ be a connected, non-compact semisimple Lie group with finite
centrer, and let $K$ be its maximal compact subgroup. Let $\Gamma$ be a
discrete subgroup of $G$ such that the quotient $\Gamma \setminus
G$ is compact. Then the quotient space $X=\Gamma \setminus G / K$
is also compact.

On page 160 of \cite{Ga68}, the author presents a formula for the heat kernel on $G$.  In
general terms, the heat kernel $K_{G}(g;t) $ with singularity when $g$ is the identity,
is equal to the inverse spherical transform of a Gaussian; see, Proposition 3.1 as well as
\cite{JLa01} in the case $G = \textrm{\rm SL}_{n}(\mathbb R)$.

 In the case that $G$ is complex, the inverse
transform can be computed and the resulting formula is particularly elementary;
see Proposition 3.2 of \cite{Ga68}.  In this case, one has that $\rho_{0}^{2}$
is equal to the norm of $1/2$ of the sum of the positive roots of the Lie algebra of $G$.

The heat kernel on $X$ can be written, as in the notation of (4.2) of \cite{Ga68},
as the series
$$
K_{X}(z,w;t) = \sum\limits_{\gamma \in \Gamma}K_{G}(z^{-1}\gamma w;t).
$$
The expressions from Proposition 3.1 and Proposition 3.2 of \cite{Ga68} are such that
the integral in (\ref{e.DefPoissonKernel}) can be computed term-by-term.  As discussed
in section 2.5, one should replace $t$ by $t/(4\rho_{0}^{2})$ so then one has the Kronecker
limit theorem as in Remark \ref{r.E_KLF}.  One can be optimistic  that the case of
general $G$ will not be significantly different from $G=\textrm{\rm SL}_{2}({\mathbb R})$.

\subsection{Concluding remarks}

Though we began with the assumption that $X$ is a K\"ahler variety, one could review
the proofs we developed and relax this condition.  For example, if $X$ is a hyperbolic
$n$-manifold, meaning the compact quotient of $\textrm{\rm SO}(n,1)$, then the structure
of the Laplacian associated to the natural hyperbolic metric is such that all aspects of
our proofs apply.  In this case, the Kronecker limit function associated to $G_{X;\rho_{0}^{2}}(z,w;s)$
would be a harmonic form with a singularity when $z$ approaches $w$.  Furthermore, the
heat kernel on the hyperbolic $n$-space has a particularly elementary expression; see, for example,
\cite{DGM76} who attribute the result to Millson.  In this case, $\rho_{0}^{2} \neq 0$, so then
would expect, as in the case when $n=2$, a generalization of the elliptic Eisenstein series as
a sum over the uniformizing group.  The study of Poincar\'e series associated to $\textrm{\rm SO}(n,1)$
is developed in \cite{CLPS91}, and it will be interested to connect those results with the
non-$L^{2}$ series $E_{X;\rho_{0}^{2}}(z,w;s)$.

Finally, we began with the heat kernel acting on smooth functions.  Certainly, one could
follow the same construction when using a form-valued heat kernel.  By doing so, one would
perhaps not consider the resolvent kernel, but rather focus on $K_{X;\rho_{0}^{2}}(z,w;s)$.
In this case, one would integrate one of the variables over a cycle $\gamma$ on $X$, as in section 5 of
\cite{JvPS16}, and study the resulting Kronecker limit function.  It seems plausible to
expect that in this manner one would obtain a direct generalization of \cite{KM79}, whose
series admitted a Kronecker limit function which was the Poincar\'e dual to the $\gamma$.

\noindent
James W. Cogdell \\
Department of Mathematics \\
Ohio State University \\
231 W. 18th Ave. \\
Columbus, OH 43210 \\
U.S.A. \\
e-mail: cogdell@math.ohio-state.edu

\vspace{5mm}\noindent
Jay Jorgenson \\
Department of Mathematics \\
The City College of New York \\
Convent Avenue at 138th Street \\
New York, NY 10031\\
U.S.A. \\
e-mail: jjorgenson@mindspring.com

\vspace{5mm}\noindent
Lejla Smajlovi\'c \\
Department of Mathematics \\
University of Sarajevo\\
Zmaja od Bosne 35, 71 000 Sarajevo\\
Bosnia and Herzegovina\\
e-mail: lejlas@pmf.unsa.ba

\end{document}